\documentclass[11pt]{amsart}
\usepackage[english]{babel}
\usepackage{amsmath}
\usepackage{changes}
\usepackage{amsthm}
\usepackage{amsfonts} 
\usepackage{hyperref}
\usepackage{epsfig}
\usepackage{amssymb}
\usepackage[margin=3cm]{geometry}
\usepackage{mathrsfs}
\DeclareMathOperator*{\argmin}{arg\,min}

\numberwithin{equation}{section}

\allowdisplaybreaks

\def\r{\varrho}

\newcommand{\eqa}{\begin{eqnarray}}
\newcommand{\ena}{\end{eqnarray}}
\newcommand{\eq}{\begin{equation}}
\newcommand{\en}{\end{equation}}
\newcommand{\eqs}{\begin{eqnarray*}}
\newcommand{\ens}{\end{eqnarray*}}

\def\X{\mathbf{X}}

\def\r{\varrho}

\newcommand{\N}     {\mathbb{N}}

\def\1{{\mathchoice {1\mskip-4mu\mathrm l}      
{1\mskip-4mu\mathrm l} 
{1\mskip-4.5mu\mathrm l} {1\mskip-5mu\mathrm l}}} 
\newcommand{\ssup}[1] {{{\scriptscriptstyle{({#1}})}}} 
\def\comment#1{}

\def\ignore#1{}

\def\C{\mathcal{C}}

\def\bP{\mathbf{P}}

\newtheorem{theorem}{Theorem}
\newtheorem{proposition}[theorem]{Proposition}
\newtheorem{lemma}[theorem]{Lemma}
\newtheorem{remark}[theorem]{Remark}

\newtheorem{corollary}[theorem]{Corollary}
\newtheorem*{ack}{Acknowledgement}

\renewcommand{\epsilon}{\varepsilon}

\title[Once-excited random walk]{Phase transition for the once-excited random walk on general trees
 }
\date{}
\author[C.-B.~Huynh]{Cong Bang Huynh}
\address{Cong Bang Huynh\\ Univ. Grenoble Alpes, CNRS, Institut Fourier, F-38000 Grenoble, France}
\email{cong-bang.huynh@univ-grenoble-alpes.fr}

\keywords{Self-interacting random walks, excited random walk, cookie random walk, recurrence, transience, branching number, branching-ruin number}

\begin{document}

\begin{abstract}
The phase transition of $M$-digging random on a general tree was studied by Collevecchio, Huynh and Kious \cite{collevecchio2018branching}. In this paper, we study particularly the critical $M$-digging random walk on a superperiodic tree that is proved to be recurrent.\\ 
We keep using the techniques introduced by Collevecchio, Kious and Sidoravicius  \cite{collevecchio2017branching} with the aim of investigating the phase transition of Once-excited random walk on general trees.\\
In addition, we prove if $\mathcal T$ is a tree whose branching number is larger than $1$, any multi-excited random walk on $\mathcal{T}$ moving, after excitation, like a simple random walk is transient. 
\end{abstract}

\maketitle

\section{Introduction}\label{s:intro}
In this paper, we study a particular case of multi-excited random walks on trees, introduced by Volkov \cite{volkov2003excited}, called the once-excited random walk.\\
Let $M\in \mathbb{N}$, $(\lambda_1,..., \lambda_M)\in (\mathbb{R}_+)^M$ and $\lambda>0$. Let $\mathcal{T}$ be an infinite, locally-finite, tree  rooted at $\r$. The $(\lambda_1,..., \lambda_M,\lambda)$-ERW on $\mathcal{T}$, is a nearest-neighbor random walk $(X_n)$ started at $\r$ such that if $X_n$ is on a site for the $i$-th time for $i\leq M$, then the walker takes a random step of a biased random walk with bias $\lambda_i$ (i.e.~it jumps on its parent with probability proportional to $1$, or jumps on a particular offspring of $\nu$ with probability proportional to $\lambda_i$); and if $i>M$, then $X_n$ takes a random step of a biased random walk with bias $\lambda$. In the case $M=1$, it is called the once-excited random walk with parameters $(\lambda_1, \lambda)$. We write $(\lambda_1,\lambda)$-OERW for $(\lambda_1, \lambda)$-ERW. The definition of the model and the vocabulary will be made clear in Section \ref{definitionofmodel}.\\

Unlike the case of once-reinforced random walk in \cite{collevecchio2017branching} or digging-random walk in \cite{collevecchio2018branching}, the phase transition of OERW does not depend only on the branching-ruin number and the branching number of tree (see Section~\ref{section:example} for more details). In the case $\mathcal {T}$ is a \emph{spherically symmetric} tree, we give a sharp phase transition recurrence/transience in terms of their {\it branching number} and {\it branching-ruin number} and others.\\

In the following, we denote $br(\mathcal T)$ the branching number of a tree $\mathcal{T}$ and $br_r(\mathcal T)$ the branching-ruin number of a tree $\mathcal{T}$, see  \eqref{branchingdef} and \eqref{branchingdef2} for their definitions. Let us simply emphasize that, for any tree $\mathcal{T}$, its branching number is at least one, i.e.~$br(\mathcal{T})\ge1$, whereas the branching-ruin number is nonnegative, i.e.~$br_r(\mathcal{T})\ge0$. 

A tree $\mathcal T$ is  said to be spherically symmetric if for
every vertex $\nu$, $\deg \nu$ depends only on $\left | \nu \right |$,
where $\left  | \nu \right  |$ denote its  distance from the  root and
$\deg \nu$  is its number  of neighbors. Let $\mathcal T$ be a spherically symmetric tree. For any $n\geq 0$, let $x_n$ be the number of children of a vertex at level $n$. For any $\lambda_1\geq 0$ and $\lambda> 0$, we define the following quantities: 
\begin{equation}
    \label{equaofalpha}
    \alpha(\mathcal{T},\lambda_1,\lambda)=\liminf_{n\rightarrow\infty}\left(\prod_{i=1}^{n}\frac{\lambda^2+(x_i-1)\lambda_1\lambda+\lambda_1}{1+x_i\lambda_1}\right)^{1/n}.
\end{equation}
\begin{equation}
    \label{equaofbeta}
    \beta(\mathcal{T},\lambda_1,\lambda)=\limsup_{n\rightarrow\infty}\left(\prod_{i=1}^{n}\frac{\lambda^2+(x_i-1)\lambda_1\lambda+\lambda_1}{1+x_i\lambda_1}\right)^{1/n}.
\end{equation}

\begin{equation}
    \label{equaofgamma}
    \gamma(\mathcal{T},\lambda_1)=\liminf_{n\rightarrow\infty}\frac{-\sum_{i=1}^{n}\ln\left [1-\frac{(x_i-1)\lambda_1+2}{(1+x_i\lambda_1)i}\right]}{\ln n}.
\end{equation}

\begin{equation}
    \label{equaofeta}
    \eta(\mathcal{T},\lambda_1)=\limsup_{n\rightarrow\infty}\frac{-\sum_{i=1}^{n}\ln\left [1-\frac{(x_i-1)\lambda_1+2}{(1+x_i\lambda_1)i}\right]}{\ln n}.
\end{equation}

\begin{theorem}
\label{th:once-excited}
Let $\mathcal T$ be a spherically symmetric tree, and let $\lambda_1\geq 0$, $\lambda>0$. Denote ${\bf X}$ the  $(\lambda_1,\lambda)$-OERW on $\mathcal{T}$. Assume that there exists a constant $M>0$ such that $\sup_{\nu\in V}\deg \nu\leq M$, then we have 
\begin{enumerate}
\item in the case $\lambda=1$,   if  $\eta(\mathcal T, \lambda_1)<br_r(\mathcal T)$ then ${\bf X}$ is transient and if $\gamma(\mathcal T, \lambda_1)>br_r(\mathcal T)$ then ${\bf X}$ is recurrent;
\item assume that $\lambda_1\geq 0$, $\lambda\neq 1$ and $br(\mathcal{T})>1$, if  $\beta(\mathcal{T}, \lambda_1, \lambda)<\frac{1}{br(\mathcal T)}$ then ${\bf X}$ is recurrent and if $\alpha(\mathcal{T}, \lambda_1, \lambda)>\frac{1}{br(\mathcal T)}$ then ${\bf X}$ is transient. 
\end{enumerate}
\end{theorem}

Note that, for a $b$-ary tree, we have $br(\mathcal{T})=b$ and 

\begin{equation}
\alpha(\mathcal{T}, \lambda_1, \lambda)=\beta(\mathcal{T}, \lambda_1, \lambda)=\frac{\lambda^2+(b-1)\lambda\lambda_1+\lambda_1}{1+b\lambda_1}
\end{equation}
and our result therefore agrees with Corollary 1.6 of \cite{basdevant2009recurrence}. In \cite{basdevant2009recurrence}, the authors prove that the walk is recurrent at criticality on regular trees, but this is not expected to be true on any tree). For instance, if $\lambda_1=\lambda$, the $(\lambda,\lambda)$-OERW $\bf X$ is the biased random walk with parameter $\lambda$. Therefore $\bf X$ may be recurrent or transient at criticality (see~\cite{beffara2017trees}, proposition 22).\\

Volkov \cite{volkov2003excited} conjectured that, any cookie random walk which moves, after excitation, like a simple random walk (i.e. $\lambda=1$) is transient on any tree containing the binary tree. This conjecture was proved by Basdevant and Singh \cite{basdevant2009recurrence}. Here, we extend this conjecture to any tree $\mathcal T$ whose branching number is larger than $1$:

\begin{theorem}
\label{thm:once-excitedbis}
Let $(\lambda_1,...,\lambda_M)\in (\mathbb{R}_+)^M$ and consider $(\lambda_1,...,\lambda_M,1)$-ERW $\bf X$ on an infinite, locally finite, rooted tree $\mathcal T$. If $br(\mathcal T)> 1$, then $\bf X$ is transient. 
\end{theorem}

The techniques used our paper rely on the strategy adopted in \cite{collevecchio2017branching} or \cite{collevecchio2018branching}. In particular, for the proof of transience, we here too view the set of edges crossed by ${\bf X}$ before returning to $\r$ as the cluster of the root in a particular correlated percolation.\\
There are two key ingredients that allow us to use the rest of the strategy from \cite{collevecchio2017branching}. First, we need to define {\it extensions} of ${\bf X}$, which are a family of coupled continuous-time versions of ${\bf X}$ defined on subtrees of $\mathcal{T}$. As in \cite{collevecchio2017branching}, we do this through Rubin's construction in Section \ref{rubin}. But we will see in Section \ref{rubin}, this construction is actually very different to a once-reinforced random walk in \cite{collevecchio2017branching} or $M$-digging random walk in \cite{collevecchio2018branching}.\\
Second, we need to prove that the correlated percolation mentioned above is in fact a {\it quasi-independent} percolation, see Lemma \ref{lemma2}. From there, the problem boils down to proving that a certain quasi-independent percolation is supercritical. \\
We refer to Theorem~\ref{maintheorem} for the more general result on a general tree. 

\section{The model}

First, we review some basic definitions of graph theory and then we define the model of multi-excited random walk on trees which was introduced by Volkov\cite{volkov2003excited} and then made general by Basdevant and Singh\cite{basdevant2009recurrence}. 

\subsection{Notation}
\label{sub:notation}
 Let  $\mathcal{T}=(V,E)$ be an infinite, locally finite, rooted tree with the root $\varrho$.\\
 Given two vertices $\nu,\mu$ of $\mathcal{T}$, we say that $\nu$ and $\mu$ are  \emph{neighbors}, denoted $\nu\sim \mu$, if $\{\nu,\mu\}$ is an edge of $\mathcal{T}$.\\
Let $\nu, \mu \in V\setminus\{\varrho\}$, the \emph{distance} between $\nu$ and $\mu$, denoted by $d(\nu,\mu)$, is the minimum number of edges of the unique self-avoiding paths joining $x$ and $y$. The distance between $\nu$ and $\varrho$ is called \emph{height} of $\nu$, denoted by $|\nu|$. The \emph{parent} of $\nu$ is the vertex $\nu^{-1}$ such that $\nu^{-1}\sim \nu$ and $|\nu^{-1}|=|\nu|-1$. We also call $\nu$ is a \emph{child} of $\nu^{-1}$.\\
For any $\nu\in V$, denote by $\partial(\nu)$ the number of children of $\nu$ and $\{\nu_1,..., \nu_{\partial \nu}\}$ is the set of children of $\nu$. We define an order on $\mathcal T$ by the following way. For all $\nu$ and $\mu$, we say that $\nu \leq \mu$ if the unique self-avoiding path joining $\varrho$ and $\mu$ contains $\nu$,  and we say that $\nu< \mu$ if moreover $\nu \neq \mu$.\\
Denote by $\mathcal{T}_n$ the set of vertices of $\mathcal{T}$ at height $n$. For any $\nu\in \mathcal{T}$, denote by $\mathcal{T}^{\nu}$ the biggest sub-tree of $\mathcal{T}$ rooted at $\nu$, i.e. $\mathcal{T}^u= \mathcal{T}[V^u]$, where $$ V^u:=\left \{ v\in  V(T): u \leq v \right \}.$$\\
For any edge $e$ of $\mathcal{T}$,
denote by  $e^+$ and $e^-$ its endpoints with $|e^+|=|e^-|+1$, and we define the \emph{height} of $e$ as $|e|=|e^+|$.\\
For two edges $e$ and $g$ of $\mathcal{T}$, we write $g\leq e$ if $g^+ \leq e^+ $ and $g<e$ if moreover $g^+\neq e^+$. For two vertices $\nu$ and $\mu$ of $\mathcal{T}$ such that $\nu<\mu$, we denote by $[\nu,\mu]$  the unique self-avoiding path joining $\nu$ to $\mu$. For two neighboring vertices $\nu$ and $\mu$, we use the slight abuse of notation $[\nu,\mu]$ to denote the edge with endpoints $\nu$ and $\mu$ (note that we allow $\mu<\nu$).\\
 For two edges $e_1$ and $e_2$ of $E$, denote by $e_1\wedge e_2$ the vertex with maximal distance from the root such that $e_1\wedge e_2\leq e_1^+$ and $e_1\wedge e_2\leq e_2^+$.\\ 

Finally, we define a particular class of trees, which is called \emph{superperiodic tree}. Let $\mathcal{T}_1=(V_1,E_1)$ and $\mathcal{T}_2=(V_2,E_2)$ be two trees. A \emph{morphism} of $\mathcal{T}_1$ to $\mathcal{T}_2$ is a map $f: \mathcal{T}_1 \rightarrow \mathcal{T}_2$ such that whenever $\nu$ and $\mu$ and $\mu$ are incident in $\mathcal{T}_1$, then so are $f(\nu)$ and $f(\mu)$ in $\mathcal{T}_2$.\\
 Let $N\geq 0$.  An infinite, locally finite and rooted tree $\mathcal T$ with the root $\varrho$, is said to be \emph{$N$-superperiodic} if for every $\nu \in V(\mathcal T)$, there exists
    an injective morphism $f:\mathcal T\rightarrow  \mathcal T^{f(o)}$ with $f(o)\in \mathcal T^\nu$ and 
    $|f(o)|-|\nu|\leq N$. A tree $\mathcal{T}$ is called \emph{superperiodic tree} if there exists $N\geq 0$ such that it is $N$-superperiodic.

\subsection{Some quantities on trees}
In this section, we review the definitions of branching number, growth rate and branching-ruin number. We refer the reader to (\cite{Furs} , \cite{LP:book}) for more details on the branching number and growth rate and \cite{collevecchio2017branching} for more details on the branching-ruin number.\\

In order to define the branching number and the branching-ruin number of a tree, we will need the notion of {\it cutsets}.\\
Let $\mathcal T$ be an infinite, locally finite and rooted tree. A cutset in $\mathcal{T}$  is a set $\pi$ of edges such that every infinite simple path from $a$ must include an edge in $\pi$. The set of cutsets is denoted by $\Pi$.\\

The branching number of $\mathcal{T}$ is defined as
\begin{equation}\label{branchingdef}
br(\mathcal T)  =  \sup \left\{  \gamma  >0 :  \inf_{\pi\in\Pi} \ \sum_{e\in
      \pi}\gamma^{-\left  |  e \right  |}>0  \right  \}\in[1,\infty].
\end{equation}
The branching-ruin number of $\mathcal{T}$ is defined as
\begin{equation}\label{branchingdef2}
br_r(\mathcal T)  =  \sup \left\{  \gamma  >0 :  \inf_{\pi\in\Pi} \ \sum_{e\in
      \pi}|e|^{-\gamma}>0  \right  \}\in[0,\infty].
\end{equation}

These quantities depend on the structure of the tree. If $\mathcal{T}$ is spherically symmetric, then there is really no information in the tree than that contained in the sequence $\left(|\mathcal T_n|, n\geq 0\right)$. Therefore, a tree which is spherically symmetric and whose $n$ generation grows like $b^n$ (resp.~$n^b$), for $b\ge1$, has a branching number (resp.~branching-ruin number) equal to $b$. For more general trees, this becomes more complicated. In the other word, there exists a tree whose $n$ generation grows like $b^n$ (resp.~$n^b$), for $b\ge1$, but its branching number (resp.~branching-ruin number) is not equal to $b$. For instance, the tree 1-3 in (\cite{LP:book}, page 4) is an example.\\

Finally, we review the definition of \emph{growth rate} of an infinite, locally finite and rooted tree $\mathcal{T}$. Define the \emph{lower growth rate} of $\mathcal{T}$ by 
\begin{equation}
\underline{gr(\mathcal T)}=\liminf \left | \mathcal T_{n} \right |^{\frac{1}{n}}.
\end{equation}

Similarly, we can define \emph{upper growth rate} of $\mathcal{T}$ by 
\begin{equation}
\overline{gr(\mathcal T)}=\limsup\left | \mathcal T_{n} \right |^{\frac{1}{n}}.
\end{equation}
In the case $\overline{gr(\mathcal T)}=\underline{gr(\mathcal T)}$, we define the \emph{growth rate} of $\mathcal{T}$, denoted by $gr(\mathcal T)$,  by taking the common value of $\overline{gr(\mathcal T)}$ and $\underline{gr(\mathcal T)}.$

Now, we state a relationship between the branching number and growth rate of a superperiodic tree. 

\begin{theorem}[see~\cite{LP:book}] \label{sousperiodic}
  Let $\mathcal T$ be a $N$-superperiodic tree   with
  $\overline{gr}(\mathcal  T)  <  \infty$.  Then the  growth  rate  of
  $\mathcal T$ \!\!exists and $gr(\mathcal T)=br(\mathcal T)$. Moreover, we have $|\mathcal{T}_n|\leq gr(\mathcal{T})^{n+N}$. 
\end{theorem}

\subsection{Definition of the model}
\label{definitionofmodel}
Now, we define the model of multi-excited random walk on trees. Let $\mathcal{C}=(\lambda_1,...,\lambda_M; \lambda)\in \left(\mathbb{R}_+\right)^M\times \mathbb{R}_+^*$ and $\mathcal{T}=(V,E)$ be an infinite, locally finite and rooted tree with the root $\varrho$. A $\mathcal{C}$ multi-excited random walk is a stochastic process $\bf X$ $:=(X_n)_{n\geq 0}$ defined on some probability space, taking the values in $\mathcal{T}$ with the transition probability defined by:

$$\mathbb{P}(X_0=\varrho)=1,$$

$$\mathbb{P}\left(X_{n+1}=(X_n)_i|X_0,\cdots, X_n\right)=\left\{\begin{matrix}
\frac{\lambda_j}{1+\partial(X_n)\lambda_j} \, \, \, \, \text{ if } j\leq M\\ 
\frac{\lambda}{1+\partial(X_n)\lambda} \, \, \, \, \text{ if } j > M
\end{matrix}\right.$$
    
$$\mathbb{P}\left(X_{n+1}=X_n^{-1}|X_0,\cdots, X_n\right)=\left\{\begin{matrix}
\frac{1}{1+\partial(X_n)\lambda_j} \, \, \, \, \text{ if } j\leq M\\ 
\frac{1}{1+\partial(X_n)\lambda} \, \, \, \, \text{ if } j > M
\end{matrix}\right.$$

where $i\in \{1,\cdots, k\}$ and $j=|\{0\leq k\leq n: X_k=X_n\}|$. 

We have some particular cases: 
\begin{itemize}
\item If $\mathcal{C}=(0,...,0; \lambda)$, then $\mathcal{C}$ multi-excited random walk is $M$-digging random walk with parameter $\lambda$ ($M$-DRW$_\lambda$), which was studied in \cite{collevecchio2018branching}.
\item If $M=0$, then $\mathcal{C}$ multi-excited random walk is the biased random walk with parameter $\lambda$, which was studied in \cite{lyons1990random}.
\item If $\mathcal{C}=(\lambda_1; \lambda)$, then $\mathcal{C}$ multi-excited random walk is $(\lambda_1,\lambda)$-OERW. 
\end{itemize}

The \emph{return time} of $\bf X$ to a vertex $\nu$ is defined by:
\begin{equation}
T(\nu):=\inf\{n\geq 1: X_n=\nu\}.
\end{equation}
We say that $\bf X$ is \emph{transient} if 
\begin{equation}
\mathbb{P}\left(T(\varrho)=\infty\right)>0. 
\end{equation}
Otherwise, we say that $\bf X$ is \emph{recurrent}. 

\section{Main results}

\subsection{Main results about Once-excited random walk}
Let $\lambda_1\geq 0$ and $\lambda>0$ and we consider the model $(\lambda_1,\lambda)$-OERW on an infinite, locally finite and rooted tree $\mathcal{T}$. First, we define the following functions. For any $e\in E$,  we set $\psi(e,\lambda)=1$ and  $\phi(e,\lambda_1,\lambda)=1$ if $|e|=1$ and, for any $e\in E$ with $|e|>1$, we set
\begin{equation}
\label{psi}
\begin{split}
    \psi(e,\lambda)&=\frac{\lambda^{|e|-1}-1}{\lambda^{|e|}-1}  \text{ if }\lambda\neq1,\\
        \psi(e,\lambda)&=\frac{|e|-1}{|e|}  \text{ if }\lambda=1.\\        
    \end{split}
\end{equation}

\begin{equation}
\label{phi}
 \phi(e, \lambda_1, \lambda) =  \frac{\lambda_1}{1+\partial(e^-)\lambda_1}+\frac{1}{1+\partial(e^-)\lambda_1}\psi(e,\lambda)\psi(e^{-1},\lambda)+\frac{(\partial(e^-)-1)\lambda_1}{1+\partial(e^-)\lambda_1}\psi(e,\lambda)   
\end{equation}
Finally, for any $e\in E$, we define: 
 
\begin{equation}
    \label{Psi}
    \Psi(e, \lambda_1, \lambda)= \prod_{g\leq e}\phi(g,\lambda_1,\lambda). 
\end{equation}
We refer the reader to Lemma~\ref{lem:lemma10} for the probabilistic interpretation of these functions.\\

In the following, we assume that
\begin{equation}\label{condition1}
  \exists M\in \mathbb{N} \text{ such that } \sup\{\deg \nu: \nu \in V\}\leq M. 
\end{equation}

Let us define the quantity $RT(\mathcal{T},{\bf X})$ which was introduced in \cite{collevecchio2017branching}:
\begin{equation}
RT(\mathcal{T},{\bf X})=\sup\{\gamma > 0: \inf_{\pi \in \Pi}\sum_{e\in \pi}\left(\Psi(e)\right)^\gamma>0 \}.
\end{equation}

\begin{theorem}
  \label{maintheorem}
  Consider an $(\lambda_1, \lambda)$-OERW on an infinite, locally finite, rooted tree $\mathcal T$, with parameters $\lambda_1\geq 0$ and $\lambda>0$. If $RT(\mathcal{T},{\bf X})<1$ then ${\bf X}$ is recurrent. If $RT(\mathcal{T},{\bf X})>1$ and if \eqref{condition1} holds, then ${\bf X}$ is transient. 
\end{theorem}

In the following, we consider  the case $\mathcal{T}$ is spherically symmetric. 

\begin{lemma}\label{mainlemma}
Consider a $(\lambda_1, \lambda)$-OERW ${\bf X}$ on a spherically symmetric $\mathcal T$, with parameters $\lambda_1\geq 0$ and $\lambda>0$. Assume that there exists a constant $M>0$ such that $\sup_{\nu\in V}\deg \nu \leq M$. We have that
\begin{enumerate}
\item in the case $\lambda=1$, if  $\eta(\mathcal T, \lambda_1)<br_r(\mathcal T)$ then $RT(\mathcal{T},{\bf X})>1$ and if $\gamma(\mathcal T, \lambda_1)>br_r(\mathcal T)$ then $RT(\mathcal{T},{\bf X})<1$;
\item assume that $\lambda_1\geq 0$, $\lambda\neq 1$ and $br(\mathcal{T})>1$, if  $\beta(\mathcal{T}, \lambda_1, \lambda)<\frac{1}{br(\mathcal T)}$ then $RT(\mathcal{T},{\bf X})<1$ and if $\alpha(\mathcal{T}, \lambda_1, \lambda)>\frac{1}{br(\mathcal T)}$  then $RT(\mathcal{T},{\bf X})>1$.
\end{enumerate}
\end{lemma}

Note that Theorem \ref{th:once-excited} is  a consequence of Theorem \ref{maintheorem} and Lemma \ref{mainlemma}.

\subsection{Main results about critical $M$-Digging random walk}
Let $M\in \mathbb{N}^*$, $\lambda>0$ and we consider the model $M$-DRW$_\lambda$ on an infinite, locally finite and rooted tree $\mathcal{T}$. In \cite{collevecchio2018branching}, Collevecchio-Huynh-Kious was proved that there is a phase transition with respect to the parameter $\lambda$, i.e there exists a critical parameter $\lambda_c$. A natural question that arises: what happens if $\lambda=\lambda_c$? As we said in the introduction, there is no a good answer for this question.\\

In \cite{basdevant2009recurrence}, Basdevant-Singh proved the critical $M$-digging random walk is recurrent on the regular trees. In this paper, we prove the critical $M$-digging random walk is still recurrent on a particular class of trees which contains the regular trees. 

\begin{theorem}
\label{thm:criticaldigging}
Let $M\in \mathbb{N}^*$ and $\mathcal{T}$ be a superperiodic tree whose upper-growth rate is finite. Then the  critical $M$-digging random walk on $\mathcal{T}$ is recurrent. 
\end{theorem}

\section{An example}
\label{section:example}

In this section, we give an example to prove that the phase transition of once-excited random walk $(\lambda_1, \lambda)-OERW$ on a tree $\mathcal T$ does not depend only on the branching-ruin number and the branching number of $\mathcal{T}$.
\bigskip

If $\mathcal{T}$ is a spherically symmetric tree, recall that $x_n(\mathcal{T})$ is the number of children of a vertex at level $n$. 

Let $\mathcal{T}$ (resp. $\mathcal{\widetilde{T}}$) be a spherically symmetric such that for any $n\geq 0$, we have $x_n(\mathcal{T})=2$ (resp. $x_n(\mathcal{\widetilde{T}})=1$ if $n$ is odd and $x_n(\mathcal{\widetilde{T}})=4$ if not). Then we obtain :

\begin{align}
br(\mathcal{T})=br(\mathcal{\widetilde{T}})=2.\\
br_r(\mathcal{T})=br_r(\mathcal{\widetilde{T}})=\infty.
\end{align}

\begin{lemma}
Consider a $(1, (\sqrt{3}-1)/2)$-OERW \,$\bf X$ (resp. $\bf \widetilde{X}$) on $\mathcal{T}$ (resp. $\mathcal {\widetilde{T}}$). Then $\bf X$ is recurrent, but $\bf \widetilde{X}$ is transient. 
\end{lemma}
\begin{proof}
Note that $\mathcal{T}$ is a binary tree, then we can apply Corollary 1.6 of \cite{basdevant2009recurrence} to imply that $\bf X$ is recurrent. On the other hand, by a simple computation we have
\begin{equation}
\label{equ:aaabbbccc1}
\alpha\left(\mathcal {\widetilde{T}},1,\frac{\sqrt{3}-1}{2}\right)=\beta\left(\mathcal {\widetilde{T}},1,\frac{\sqrt{3}-1}{2}\right)>\frac{1}{2}.
\end{equation}
By Theorem~\ref{th:once-excited} and \ref{equ:aaabbbccc1}, we obtain $\bf \widetilde{X}$ is transient. 
\end{proof}

\section{Proof of Theorem \ref{thm:once-excitedbis}}
\begin{lemma}
\label{lem:br_rT=infinity}
Let $\mathcal{T}$ be an infinite, locally finite and rooted tree. If $br(\mathcal{T})>1$ then $br_r(\mathcal{T})=+\infty$. 
\end{lemma}

\begin{proof}
See (\cite{collevecchio2018branching}, proof of Lemma 8, Case V). 
\end{proof}

\begin{lemma}
Let $(\lambda_1,...,\lambda_M)\in \left(\mathbb{R}_+\right)^M$ and $\mathcal{T}$ be an infinite, locally finite and rooted tree. If $M$-DRW$_{1}$ is transient, then $(\lambda_1,...,\lambda_M,1)$-ERW is transient. 
\end{lemma}

\begin{proof}
See (\cite{basdevant2009recurrence}, Section 3).
\end{proof}

\begin{remark}
\label{rem:abcd1}
Let $T(\varrho)$ (resp. $S(\varrho)$) the return of of  $M$-DRW$_{1}$ (resp. $(\lambda_1,...,\lambda_M,1)$-ERW) to the root $\varrho$ of $\mathcal{T}$. It is simple to see that 
\begin{equation}
\mathbb{P}(T(\varrho)<\infty)\leq \mathbb{P}(S(\varrho)<\infty). 
\end{equation}
\end{remark}

\begin{proposition}
\label{prop:once-excitedbis}
Let $(\lambda_1,...,\lambda_M)\in (\mathbb{R}_+)^M$ and consider $(\lambda_1,...,\lambda_M,1)$-ERW $\bf X$ on an infinite, locally finite, rooted tree $\mathcal T$. If $br(\mathcal T)> 1$, then $\bf X$ is transient. 
\end{proposition}

\begin{proof}
Note that if $\lambda_i=0$ for all $1\leq i\leq M$ and $\lambda=1$, then $\bf X$ is a $M$-digging random walk with parameter $1$ ($M$-DRW$_{1}$). On the other hand, we have  $(\lambda_1,...,\lambda_M,1)$-ERW is more transient than $M$-DRW$_{1}$, i.e if $M$-DRW$_{1}$ is transient then $(\lambda_1,...,\lambda_M,1)$-ERW is transient. We complete the proof by using Lemma~\eqref{lem:br_rT=infinity} and Theorem 2 in \cite{collevecchio2018branching}. 
\end{proof}

\section{Proof of Lemma \ref{mainlemma} and  Theorem \ref{th:once-excited}}

In this section, we prove Lemma \ref{mainlemma}. Theorem \ref{th:once-excited} then trivially follows from Theorem \ref{maintheorem}. 
\begin{lemma}
\label{lem:computePhi}
Recall the definition of $\Psi(e,\lambda_1,\lambda)$ as in \ref{Psi}. We have that, if $\lambda\neq 1$, for any $|e|>1$, 
\begin{equation}
\label{equ:computePhi1}
 \Psi(e,\lambda_1,\lambda)=\left(\prod_{g\leq e, \, |g|>1}\frac{\lambda^2+(\partial(g^-)-1)\lambda_1\lambda+\lambda_1}{1+\partial(g^-)\lambda_1}\right) \prod_{g\leq e, \, |g|>1}\left(\frac{1-\lambda^{|g|}\left(\frac{1+\partial(g^-)\lambda_1}{\lambda^2+(\partial(g^-)-1)\lambda_1\lambda+\lambda_1}\right)}{1-\lambda^{|g|}}\right).
\end{equation}
and if $\lambda=1$, for any $|e|>1$, 
\begin{equation}
\label{equ:computePhi2}
\Psi(e,\lambda_1,\lambda)=\prod_{g\leq e, \, |g|>1}\left(1-\frac{(\partial(g^-)-1)\lambda_1+2}{|g|\left(1+\partial(g^-)\lambda_1\right)} \right).
\end{equation}

\end{lemma}

\begin{proof}
We compute the quantity $\Psi(e,\lambda, \lambda_1)$ by using \eqref{psi}, \ref{phi} and \eqref{Psi}. We will proceed by distinguishing two cases.\\
\noindent
{\bf Case I: $\lambda\neq 1$.}\\
By \eqref{psi}, \ref{phi} and \eqref{Psi}, we have

$$ \Psi(e,\lambda_1,\lambda)=\prod_{g\leq e ,\,  |g|>1}\phi(g,\lambda_1,\lambda)$$
$$=\prod_{g\leq e, \, |g|>1}\left(\frac{\lambda_1}{1+\partial(g^-)\lambda_1}+\frac{1}{1+\partial(g^-)\lambda_1}\psi(e,\lambda)\psi(e^{-1},\lambda)+\frac{(\partial(g^-)-1)\lambda_1}{1+\partial(g^-)\lambda_1}\psi(e,\lambda)\right)$$

$$=\left(\prod_{g\leq e, \, |g|>1}\frac{1}{1+\partial(g^-)\lambda_1}\right) \prod_{g\leq e, \, |g|>1}\left(\lambda_1+\psi(e,\lambda)\psi(e^{-1},\lambda)+(\partial(g^-)-1)\lambda_1\psi(e,\lambda)\right) $$

By~\ref{psi}, we have:

\begin{equation}
\begin{split}
&\lambda_1+\psi(g,\lambda)\psi(g^{-1},\lambda)+(\partial(g^-)-1)\lambda_1\psi(g,\lambda) \\
=&\lambda_1+\left(\frac{1-(1/\lambda)^{|g|-2}}{1-(1/\lambda)^{|g|}}\right)+\left((\partial(g^-)-1)\lambda_1 \frac{1-(1/\lambda)^{|g|-1}}{1-(1/\lambda)^{|g|}}\right)\\
=&\lambda_1+\left(\frac{\lambda^{|g|}-\lambda^2}{\lambda^{|g|}-1}\right)+(\partial(g^-)-1)\lambda_1 \left(\frac{\lambda^{|g|}-\lambda}{\lambda^{|g|}-1}\right)\\
{=}&\frac{\lambda^2+(\partial(g^-)-1)\lambda_1\lambda+\lambda_1-\lambda^{|g|}\left(1+\partial(g^-)\lambda_1\right)}{1-\lambda^{|g|}}  \\
=&\left(\lambda^2+(\partial(g^-)-1)\lambda_1\lambda+\lambda_1\right) \left(\frac{1-\lambda^{|g|}\left(\frac{1+\partial(g^-)\lambda_1}{\lambda^2+(\partial(g^-)-1)\lambda_1\lambda+\lambda_1}\right)}{1-\lambda^{|g|}}\right).
\end{split}
\end{equation}
Therefore we obtain~\ref{equ:computePhi1}.

\noindent

{\bf Case II: $\lambda=1$.}\\
By \eqref{psi}, \ref{phi} and \eqref{Psi}, we have

$$ \Psi(e,\lambda_1,\lambda)=\prod_{g\leq e ,\,  |g|>1}\phi(g,\lambda_1,\lambda)$$
$$=\prod_{g\leq e, \, |g|>1}\left(\frac{\lambda_1}{1+\partial(g^-)\lambda_1}+\frac{1}{1+\partial(g^-)\lambda_1}\psi(e,\lambda)\psi(e^{-1},\lambda)+\frac{(\partial(g^-)-1)\lambda_1}{1+\partial(g^-)\lambda_1}\psi(e,\lambda)\right)$$

$$=\left(\prod_{g\leq e, \, |g|>1}\frac{1}{1+\partial(g^-)\lambda_1}\right) \prod_{g\leq e, \, |g|>1}\left(\lambda_1+\psi(e,\lambda)\psi(e^{-1},\lambda)+(\partial(g^-)-1)\lambda_1\psi(e,\lambda)\right) $$
By~\ref{psi}, we have:

\begin{equation}
\begin{split}
&\lambda_1+\psi(g,\lambda)\psi(g^{-1},\lambda)+(\partial(g^-)-1)\lambda_1\psi(g,\lambda) \\
=&\lambda_1+\frac{|g|-2}{|g|}+(\partial(g^-)-1)\lambda_1 \frac{|g|-1}{|g|}\\
=&\frac{\lambda_1|g|+|g|-2+(\partial(g^-)-1)\lambda_1(|g|-1)}{g|}\\
{=}&1+\partial(g^-)\lambda_1-\frac{(\partial(g^-)-1)\lambda_1+2}{|g|}
\end{split}
\end{equation}
Therefore we obtain~\ref{equ:computePhi2}.
\end{proof}

\begin{proof}[Proof of Lemma \ref{mainlemma}]
We will proceed by distinguishing a few cases.\\

\noindent
{\bf Case I: $\lambda\neq 1$, $br(\mathcal{T})>1$ and $\beta(\mathcal{T}, \lambda_1, \lambda)<\frac{1}{br(\mathcal{T})}$.}\\
By \eqref{branchingdef},  there exists $\delta\in(0,1)$ such that
\begin{equation}\label{small}
\inf_{\pi\in\Pi} \sum_{e\in \Pi}\beta^{(1-\delta)^2|e|}=0.
\end{equation}
As $\beta<\beta^{(1-\delta)}$, there exists $c>0$, for any $n>0$,
\begin{equation}
\label{equ:abc1}
\prod_{i=1}^{n}\frac{\lambda^2+(x_i-1)\lambda_1\lambda+\lambda_1}{1+x_i\lambda_1}\leq c\, \beta^{(1-\delta)n}.
\end{equation}

By ~\ref{equ:computePhi1} and \ref{equ:abc1}, there exists $C>0$ such that for any $\pi\in\Pi$,
\begin{equation}\label{comp1}
\begin{split}
\sum_{e\in\pi} \Psi(e)^{1-\delta}&\leq C\sum_{e\in \Pi}\beta^{(1-\delta)^2|e|}.
\end{split}
\end{equation}
Therefore, by \eqref{small},
\begin{equation}
\inf_{\pi\in\Pi} \sum_{e\in\pi} \Psi(e)^{1-\delta}=0,
\end{equation}
which implies that $RT(\mathcal{T},{\bf X})<1$.\\

\noindent
{\bf Case II: $\lambda\neq 1$, $br(\mathcal{T})>1$ and $\alpha(\mathcal T,\lambda_1,\lambda)>\frac{1}{br(\mathcal{T})}$.}\\
First, note that if $\lambda>1$ and $br(\mathcal{T})>1$ then $\bf X$ is transient. Now, assume that $\lambda<1$, $br(\mathcal{T})>1$ and $\alpha(\mathcal T,\lambda_1,\lambda)>\frac{1}{br(\mathcal{T})}$. We have that there exists $\delta>0$ and $\epsilon>0$ such that
\begin{equation}\label{small222}
\inf_{\pi\in\Pi} \sum_{e\in \Pi} \alpha^{(1+\delta)^2|e|}>\epsilon.
\end{equation}
By \ref{equaofalpha} and $\lambda<1$, we obtain $\alpha<1$, therefore $\alpha^{1+\delta}<\alpha$. We have that there exists $c>0$, for any $n>0$,
\begin{equation}
\label{equ:abc2}
\prod_{i=1}^{n}\frac{\lambda^2+(x_i-1)\lambda_1\lambda+\lambda_1}{1+x_i\lambda_1}\geq c\, \alpha^{(1+\delta)n}.
\end{equation}
By ~\ref{equ:computePhi1} and \ref{equ:abc2}, there exists $C>0$ such that for any $\pi\in\Pi$,
\begin{equation}\label{comp1}
\begin{split}
\sum_{e\in\pi} \Psi(e)^{1+\delta}&\geq C\sum_{e\in \Pi}\alpha^{(1+\delta)^2|e|}.
\end{split}
\end{equation}

Therefore, by \eqref{small222},
\begin{equation}
\inf_{\pi\in\Pi} \sum_{e\in\pi} \Psi(e)^{1+\delta}>0,
\end{equation}
which implies that $RT(\mathcal{T},{\bf X})>1$.\\

\noindent
{\bf Case III: $\lambda=1$ and $\eta(\mathcal{T}, \lambda_1)<br_r(\mathcal{T})$.}\\
We have that there exists $\delta>0$ and $\epsilon>0$ such that
\begin{equation}
\label{small2222}
\inf_{\pi\in\Pi} \sum_{e\in \pi}|e|^{-(1+\delta)^2\eta}>\epsilon.
\end{equation}
As $\eta<(1+\delta)\eta$, by \ref{equaofeta} there exists $c>0$, for any $n>0$,
\begin{equation}
\label{equ:abc3}
\prod_{i=1}^{n}\left [1-\frac{(x_i-1)\lambda_1+2}{(1+x_i\lambda_1)i}\right]\geq c\, n^{-(1+\delta)\eta}.
\end{equation}
By ~\ref{equ:computePhi2} and \ref{equ:abc3}, there exists $C>0$ such that for any $\pi\in\Pi$,
\begin{equation}\label{comp1}
\begin{split}
\sum_{e\in\pi} \Psi(e)^{1+\delta}&\geq C\sum_{e\in \Pi}|e|^{-(1+\delta)^2\eta}.
\end{split}
\end{equation}
Therefore, by \eqref{small2222},
\begin{equation}
\inf_{\pi\in\Pi} \sum_{e\in\pi} \Psi(e)^{1+\delta}>0,
\end{equation}
which implies that $RT(\mathcal{T},{\bf X})>1$.\\

\noindent
{\bf Case IV:  $\lambda=1$ and $\gamma(\mathcal T, \lambda_1)>br_r(\mathcal{T})$}\\
We have that there exists $\delta>0$ such that
\begin{equation}
\label{small22222}
\inf_{\pi\in\Pi} \sum_{e\in \pi}|e|^{-(1-\delta)^2\eta}=0.
\end{equation}
As $\eta>(1-\delta)\eta$, by \ref{equaofeta} there exists $c>0$, for any $n>0$,
\begin{equation}
\label{equ:abc4}
\prod_{i=1}^{n}\left [1-\frac{(x_i-1)\lambda_1+2}{(1+x_i\lambda_1)i}\right]\leq c\, n^{-(1-\delta)\eta}.
\end{equation}
By ~\ref{equ:computePhi2} and \ref{equ:abc4}, there exists $C>0$ such that for any $\pi\in\Pi$,
\begin{equation}\label{comp1}
\begin{split}
\sum_{e\in\pi} \Psi(e)^{1-\delta}&\leq C\sum_{e\in \Pi}|e|^{-(1-\delta)^2\eta}.
\end{split}
\end{equation}
Therefore, by \eqref{small22222},
\begin{equation}
\inf_{\pi\in\Pi} \sum_{e\in\pi} \Psi(e)^{1-\delta}>0,
\end{equation}
which implies that $RT(\mathcal{T},{\bf X})<1$.
\end{proof}

\section{Extensions}
\label{rubin}
First of all, let us describe the dynamic of this model. If $\bf X$ visits a vertex $\nu$ for the first time, three cases can occur for visiting $\nu_1$ (see Figure~\ref{fig:abcdefgh1}): 
\begin{itemize}
\item It eats the cookie at $\nu$ and returns to the parent of $\nu$ (i.e.\ $\nu^{-1}$) with probability $\frac{1}{1+\partial(\nu)\lambda_1}$. It then visits $\nu$ for the second time, and goes to $\nu_1$ with probability $\frac{\lambda}{1+\partial(\nu)\lambda}$. 
\item It goes directly to $\nu_1$ with probability $\frac{\lambda_1}{1+\partial(\nu)\lambda_1}$. 
\item It goes to one of the chidren of $\nu$ except for $\nu_1$, with probability $\frac{(\partial \nu-1)\lambda_1}{1+\partial(\nu)\lambda_1}$. It then visits $\nu$ for the second time, and goes to $\nu_1$ with probability $\frac{\lambda}{1+\partial(\nu)\lambda}$. 
\end{itemize} 

\begin{figure}[h!]
    \centering
    \includegraphics[scale=0.6]{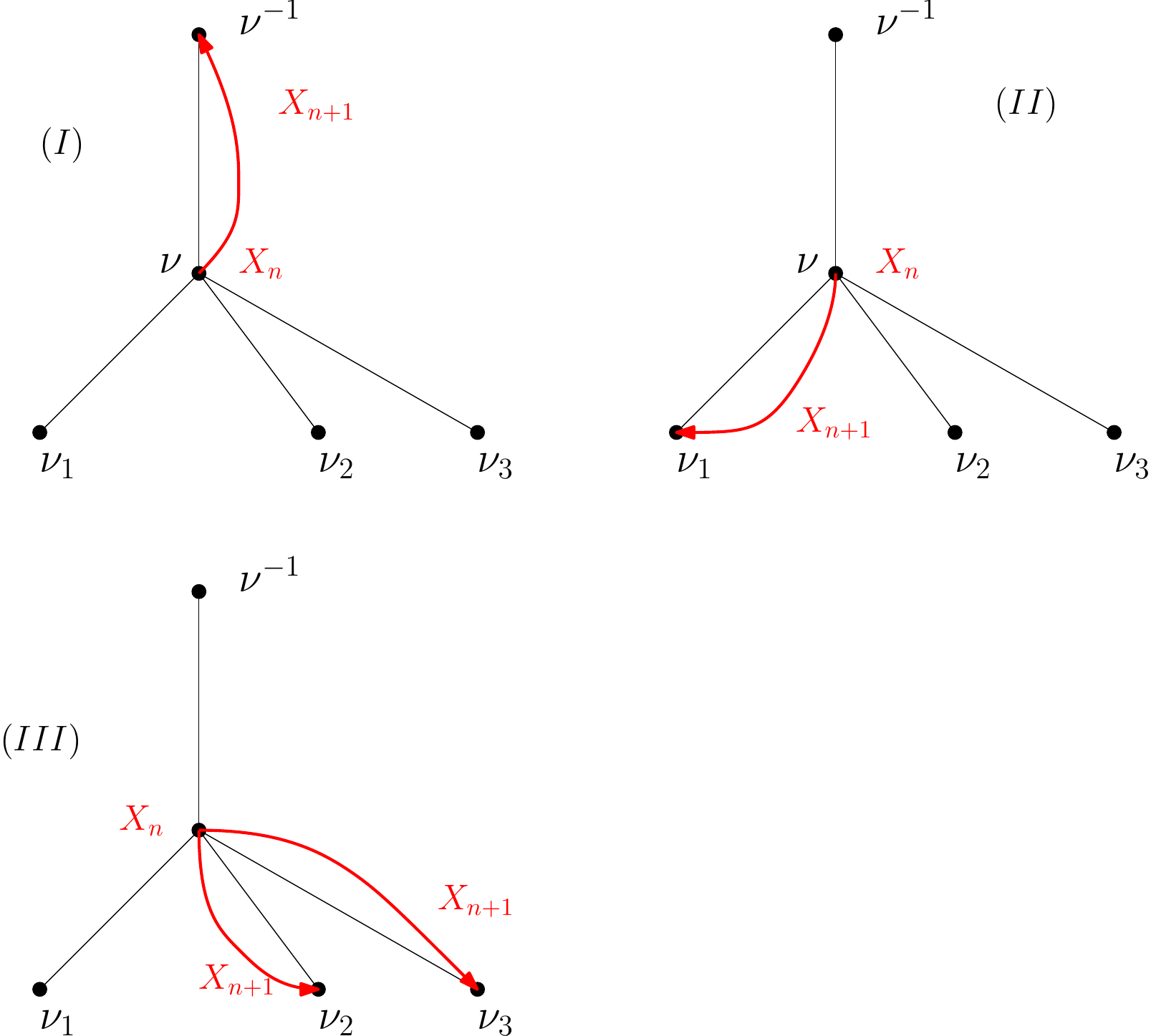} 
    \caption{The movement of $\bf X$ to $\nu_1$ after visiting $\nu$.}
    \label{fig:abcdefgh1}
  \end{figure}

\bigskip

Now, we introduce a construction of once-excited random walk by using the Rubin's construction. Let $(\Omega, \mathcal{F},\bP)$ denote a probability space on which
\begin{align}\label{defY}
{\bf Y}=(Y(\nu,\mu,k): (\nu,\mu)\in V^2, \mbox{with }\nu \sim \mu, \textrm{ and }k \in \N)\\
{\bf Z}=(Z(\nu,\mu): (\nu,\mu)\in V^2, \mbox{with }\nu \sim \mu)
\end{align}
are two families of independent mean $1$ exponential random variables, where $(\nu,\mu)$ {denotes} an {\it ordered} pair { of vertices}. Let 
\begin{align}\label{defU}
{\bf U}=(U_\nu: \nu \in V)
\end{align}
is a family of independent uniformly random variables on $[0,1]$ which is independent to $\bf Y$ and $\bf Z$. 
 
For any pair vertices   $\nu,\mu\in V$ with $\nu\sim \mu$, we define the following quantities
\begin{equation} \label{wj1}
r(\nu,\mu)=\left\{\begin{array}{cc}
&\lambda^{|\nu|-1}\text{, if }\mu<\nu,\\
&\lambda^{|\mu|-1}\text{, if }\nu<\mu.
\end{array}\right.
\end{equation}

Let $\mathcal{T}'$ be a sub-tree of $\mathcal{T}$, we define the {\it extension} ${\bf X}^{ (\mathcal{T}')}=(V',E')$  on $\mathcal{T}'$ in the following way. Denote by $\varrho'$ the root of $\mathcal{T}'$ which  be defined as the vertex of $V'$ with smallest distance to the root of $\mathcal{T}$.
For  any family of nonnegative integers $\bar{k}=(k_\mu)_{\mu: [\nu,\mu]\in E'} $, we let 
\begin{align}
A^{ (\mathcal{T}')}_{\bar{k},n,\nu}:=\{{X}^{ (\mathcal{T}')}_n = \nu\}\cap\bigcap_{\mu: [\nu,\mu]\in E'} \{\#\{1\le j \le n \colon ({X}^{ (\mathcal{T}')}_{j-1},{X}^{ (\mathcal{T}')}_j) = (\nu,\mu)\} = k_\mu\}.
\end{align}
\begin{align}
t_\nu(n):=\#\{1\leq j\leq n: X^{(\mathcal{T}')}_j=\nu\}. 
\end{align}
\begin{align}
h_\nu:=\inf\{i\geq 1: t_\nu(i)=2\}.
\end{align}
\begin{align}
\widetilde{A}^{ (\mathcal{T}')}_{\bar{k},n,\nu}:=\{{X}^{ (\mathcal{T}')}_n = \nu\}\cap\bigcap_{\mu: [\nu,\mu]\in E'} \{\#\{h_\nu\le j \le n \colon ({X}^{ (\mathcal{T}')}_{j-1},{X}^{ (\mathcal{T}')}_j) = (\nu,\mu)\} = k_\mu\}.
\end{align}
\begin{align}
\mathcal{I}^{\mathcal T}(\nu):=\#\{i\in \{1,2,\cdots, \partial(\nu)\}: \nu_i\in V(\mathcal T')\}. 
\end{align}
Set ${X}^{ (\mathcal{T}')}_0=\r'$ and on the event $A^{ (\mathcal{T}')}_{\bar{k},n,\nu}\cap \{t_\nu(n)\leq 1\}$: 
\begin{itemize}
\item If $U_\nu<\frac{1}{1+\partial(\nu)\lambda_1}$, then we set ${X}^{ (\mathcal{T}')}_{n+1}=\nu^{-1}$.
\item If \, $U_\nu\in \left[\frac{1+(j-1)\lambda_1}{1+\partial(\nu)\lambda_1}, \frac{1+j\lambda_1}{1+\partial(\nu)\lambda_1}\right]$ and $j\in \mathcal{I}^\mathcal{T}(\nu)$, then we set $X^{(\mathcal{T'})}_{n+1}=v_j$.
\item If \, $U_\nu\in \left[\frac{1+(j-1)\lambda_1}{1+\partial(\nu)\lambda_1}, \frac{1+j\lambda_1}{1+\partial(\nu)\lambda_1}\right]$ for some $j\notin \mathcal{I}^\mathcal{T}(\nu)$ and 
$$\left\{\nu' = \argmin_{\mu: [\nu,\mu]\in E'}\Big\{\frac{Z(\nu, \mu)}{r(\nu, \mu}\Big\}\right\},$$
we set $X^{(\mathcal{T'})}_{n+1}=\nu'$. 
\end{itemize}

On the event 
\begin{align}\label{ursula}
\widetilde{A}^{ (\mathcal{T}')}_{\bar{k},n,\nu}\cap \{t_\nu(n)\geq 2\}\cap \left\{\nu' = \argmin_{\mu: [\nu,\mu]\in E'}\Big\{\sum_{i=0}^{k_{\mu}}\frac{Y(\nu, \mu, i)}{r(\nu, \mu)} \Big\}\right\}, 
\end{align}
 we set ${X}^{ (\mathcal{T}')}_{n+1} = \nu'$, where the function $r$ is defined in \eqref{wj1} and the clocks $Y$'s are from the same collection ${\bf Y}$ fixed in \eqref{defY}.\\
 
Thus, this defines ${\bf X}^{(\mathcal{T})}$ as the extension on the whole tree.
  By using the properties of independent exponential random variables, it is easy to check that this construction is a construction of  $(\lambda_1,\lambda)$-OERW on the tree $\mathcal{T}$. We refer the reader to (\cite{collevecchio2018branching}, section 7) for more discussions  on this construction. 

In the case $\mathcal{T}'=[\r,\nu]$ for some vertex $\nu$ of $\mathcal{T}$, we write ${\bf X}^{(\nu)}$ instead of ${\bf X}^{([\r,\nu])}$, and we denote $T^{(\nu)}(\cdot)$ the return times associated to ${\bf X}^{(\nu)}$. For simplicity, we will also write  ${\bf X}^{(e)}$ and $T^{(e)}(\cdot)$ instead of ${\bf X}^{(e^+)}$ and $T^{(e^+)}(\cdot)$ for $e\in E$.

\begin{remark}
Let $\mathcal{T}'$ be a proper subtree of $\mathcal{T}$. Note that ${\bf X}^{(\mathcal{T}')}$ is not $(\lambda_1,\lambda)$-OERW on $\mathcal{T}'$, that is different with $M$-digging random walk (see \cite{collevecchio2018branching}, section 7) and once-reinforced random walk (see \cite{collevecchio2017branching}, section 5). 
\end{remark}

Finally, we give a probabilistic interpretation of the functions $\phi$ and $\Psi$: 
\begin{lemma}
\label{lem:lemma10}
For any $e\in E$ and any $g\le e$, we have
\begin{align}\label{eqhit1}
\phi(g,\lambda_1,\lambda)&=\mathbb P\left(T^{(e)}(g^+)\circ \theta_{T^{(e)}(g^-)}<T^{(e)}({\r})\circ \theta_{T^{(e)}(g^-)}\right),\\ \label{eqhit}
\Psi(e,\lambda_1,\lambda)&=\mathbb P\left(T^{(e)}(e^+)<T^{(e)}({\r})\right),
\end{align}
where $\theta$ is the canonical shift on the trajectories.
\end{lemma}

\begin{proof}
Let $e\in E$ and $g\le e$. For simplicity, we set
$$\mathcal A:=\{T^{(e)}(g^+)\circ \theta_{T^{(e)}(g^-)}<T^{(e)}({\r})\circ \theta_{T^{(e)}(g^-)}\},$$ 
$$\mathcal{I}_1:= \left[\frac{1+(j-1)\lambda_1}{1+\partial(g^-)\lambda_1}, \frac{1+j\lambda_1}{1+\partial(g^-)\lambda_1}\right],$$
$$\mathcal{I}_2:=[0,1]\setminus \left(\left[\frac{1+(j-1)\lambda_1}{1+\partial(g^-)\lambda_1}, \frac{1+j\lambda_1}{1+\partial(g^-)\lambda_1}\right]\bigcup \left[0, \frac{1}{1+\partial(g^-)\lambda_1}\right]\right), $$
where $j\in \{1,..., \partial(g^-)\}$ such that $(g^-)_j=g^+$. We have that
\begin{equation}\label{equ:bac0}
\begin{split}
&\mathbb P\left(\mathcal A\right)=\mathbb P\left(A\Big{|}U_{g^-}<\frac{1}{1+\partial(g^-)\lambda_1}\right)\times \mathbb{P}\left(U_{g^-}<\frac{1}{1+\partial(g^-)\lambda_1}\right) \\
+&\mathbb P\left(A|\mathcal{I}_1\right)\times \mathbb{P}\left(U_{g^-}\in \mathcal{I}_1\right)+\mathbb P\left(A|\mathcal{I}_2\right)\times \mathbb{P}\left(U_{g^-}\in \mathcal{I}_2\right).
\end{split}
\end{equation}
On the other hand, we have the following equalities:
\begin{equation}\label{equ:bac1}
\mathbb P\left(A\Big{|}U_{g^-}<\frac{1}{1+\partial(g^-)\lambda_1}\right)\times \mathbb{P}\left(U_{g^-}<\frac{1}{1+\partial(g^-)\lambda_1}\right)=\\ 
\frac{1}{1+\partial(g^-)\lambda_1}\psi(g,\lambda)\psi(g^{-1},\lambda)
\end{equation}
\begin{equation}\label{equ:bac2}
\mathbb P\left(A|\mathcal{I}_1\right)\times \mathbb{P}\left(U_{g^-}\in \mathcal{I}_1\right)=\frac{\lambda_1}{1+\partial(g^-)\lambda_1}.
\end{equation}
\begin{equation}\label{equ:bac3}
\mathbb P\left(A|\mathcal{I}_2\right)\times \mathbb{P}\left(U_{g^-}\in \mathcal{I}_2\right)=\frac{(\partial(g^-)-1)\lambda_1}{1+\partial(g^-)\lambda_1}\psi(g,\lambda).
\end{equation}
We use \eqref{equ:bac0}, \eqref{equ:bac1}, \eqref{equ:bac2} and \eqref{equ:bac3} to obtain the results. 
\end{proof}

\section{Recurrence in Theorem \ref{maintheorem}: The case $RT(\mathcal{T},{\bf X})<1$}

 \begin{proposition}\label{proprec}
 If
 \begin{align}\label{almadort}
\inf_{\pi\in \Pi}\sum_{e\in\pi}\Psi(e)=0,
\end{align}
then $\X$ is recurrent.
 \end{proposition}
 
 \begin{proof}
The proof is identical to the proof of Proposition 10 of \cite{collevecchio2017branching}.
\end{proof}

\section{Transience in Theorem \ref{maintheorem}: The case $RT(\mathcal{T},{\bf X})>1$}

In order to prove transience, we use the relationship between the walk $\bf{X}$ and its associated percolation. 
\subsection{Link with percolation}
\label{linkwithperco}

Denote by $C(\r)$ the set of edges which are {crossed} by ${\bf X}$ before returning to $\r$, that is: 
\begin{equation}
\mathcal C(\r)=\{e\in E: T(e^+)<T(\r) \}.
\end{equation}

We define an other percolation which will be more easy to study. In order to do this, we use the Rubin'€™s construction and the extensions introduced in Section~\ref{rubin}. We define
\begin{equation}
\mathcal C_{CP}(\varrho)=\{e\in E: T^{(e)}(e^+)<T^{(e)}(\varrho)\}.
\end{equation}

We say that an edge $e\in E$ is open if and only if $e\in \mathcal C_{CP}(\varrho)$.

\begin{lemma}
\label{lemma1}
We have that
\begin{equation}
\mathbb{P}(T(\r)=\infty)=\mathbb{P}(|\mathcal{C}(\r)|=\infty)=\mathbb{P}(|\mathcal{C}_{CP}(\r)|=\infty).
\end{equation}
\end{lemma}

\begin{proof}
We can follow line by line the proof of Lemma 11 in \cite{collevecchio2017branching}.
\end{proof}

For simplicity, for a vertex $v\in V$, we write $v \in\C_{\mathrm{CP}}(\r)$ if one of the edges incident to $v$ is in $\C_{\mathrm{CP}}(\r)$. Besides, recall that for two edges $e_1$ and $e_2$, their common ancestor  with highest generation is the vertex denoted $e_1\wedge e_2$.

\begin{lemma}
\label{lemma2}
Let $\lambda_1\geq 0$, $\lambda>0$ and $\mathcal{T}$ be an infinite, locally finite and rooted tree with the root $\varrho$. Assume that the condition \eqref{condition1} holds with some constant $M$. Then the correlated percolation induced by $\mathcal{C}_{CP}$ is quasi-independent, i.e.~there exists a constant $C_Q\in (0,+\infty)$ such that, for any two edges $e_1, e_2$, we have that
\begin{equation}\label{qindep0}
\begin{split}
\mathbb{P}(e_1, e_2\in \mathcal C_{CP}(\r)|e_1\wedge e_2 \in \mathcal C_{CP}(\r))\leq& C_Q \mathbb{P}(e_1\in \mathcal C_{CP}(\r)|e_1\wedge e_2 \in \mathcal C_{CP}(\r))\\
&\times \mathbb{P}(e_2\in \mathcal C_{CP}(\r)|e_1\wedge e_2 \in \mathcal C_{CP}(\r)).
\end{split}
\end{equation}
\end{lemma}
\begin{proof}
Recall the construction of Section~\ref{rubin}. Note that if $e_1\wedge e_2=\varrho$, then the extensions on $[\varrho, e_1]$ and $[\varrho, e_2]$ are independent, then the conclusion of Lemma holds with $C=1$. 
Assume that $e_1\wedge e_2 \neq \varrho$, and  note  that  the  extensions  on $[\varrho, e_1]$ and $[\varrho, e_2]$ are dependent since they use the same  clocks on 
 $[\varrho, e_1\wedge e_2]$. Denote by $e$ the unique edge of $\mathcal{T}$ such that $e^+=e_1\wedge e_2$. For $i\in \{1,2\}$, let $v_i$ be the vertex which is the offspring of $e^+$ lying the path from $\varrho$ to $e_i$. Note that $v_i$ could be equal to $e^+_i$. Let $i_1$ (resp. $i_2$) be an element of $\{1,..., \partial(e^+)\}$ such that $(e^+)_{i_1}=v_1$ (resp. $(e^+)_{i_2}=v_2$).\\
As the events $\{e\in \mathcal{C}_{CP}\}$ and $U_{e_1\wedge e_2}$ are independent, therefore: 
 $$\mathbb{P}\left(e_1, e_2\in \mathcal C_{CP}(\varrho)|e \in \mathcal C_{CP}(\varrho)\right)=A+B+C+D,$$
 where
 \begin{equation}
 A=\mathbb{P}\left(e_1, e_2\in \mathcal C_{CP}(\varrho)\Big{|}e \in \mathcal C_{CP}(\varrho),\, U_{e^+}<\frac{1}{1+\partial(e^+)\lambda_1}\right)\mathbb{P}\left(U_{e^+}<\frac{1}{1+\partial(e^+)\lambda_1}\right)
 \end{equation}
 \begin{equation}
 \begin{split}
B=\mathbb{P}\left(e_1, e_2\in \mathcal C_{CP}(\varrho)\Big{|}e \in \mathcal C_{CP}(\varrho),\, U_{e^+}\in \left[\frac{1+(i_{1}-1)\lambda_1}{1+\partial(e^+)\lambda_1}, \frac{1+i_{1}\lambda_1}{1+\partial(e^+)\lambda_1}\right]\right)\\
\times\mathbb{P}\left( U_{e^+}\in \left[\frac{1+(i_{1}-1)\lambda_1}{1+\partial(e^+)\lambda_1}, \frac{1+i_{1}\lambda_1}{1+\partial(e^+)\lambda_1}\right]\right).
 \end{split}
\end{equation}  
\begin{equation}
 \begin{split}
C=\mathbb{P}\left(e_1, e_2\in \mathcal C_{CP}(\varrho)\Big{|}e \in \mathcal C_{CP}(\varrho),\, U_{e^+}\in \left[\frac{1+(i_{2}-1)\lambda_1}{1+\partial(e^+)\lambda_1}, \frac{1+i_{2}\lambda_1}{1+\partial(e^+)\lambda_1}\right]\right)\\
\times\mathbb{P}\left( U_{e^+}\in \left[\frac{1+(i_{2}-1)\lambda_1}{1+\partial(e^+)\lambda_1}, \frac{1+i_{2}\lambda_1}{1+\partial(e^+)\lambda_1}\right]\right).
 \end{split}
\end{equation}  
 \begin{equation}
 \begin{split}
 D=\mathbb{P}\left(e_1, e_2\in \mathcal C_{CP}(\varrho)\Big{|}e \in \mathcal C_{CP}(\varrho),\, U_{e^+}\in \bigcup_{i\in \{1,\cdots,\partial(e^+)\}\setminus\{i_{1},i_{2}\}}\left[\frac{1+(i-1)\lambda_1}{1+\partial(e^+)\lambda_1}, \frac{1+i\lambda_1}{1+\partial(e^+)\lambda_1}\right]\right)\\
\times \mathbb{P}\left( U_{e^+}\in \bigcup_{i\in\{1,\cdots,\partial(v)\} \setminus \{i_{1},i_{2}\}}\left[\frac{1+(i-1)\lambda_1}{1+\partial(e^+)\lambda_1}, \frac{1+i\lambda_1}{1+\partial(e^+)\lambda_1}\right]\right).
 \end{split}
 \end{equation}
 \bigskip
 
 In the same way, for any $j\in\{1,2\}$, we have:
  $$\mathbb{P}\left(e_j\in \mathcal C_{CP}(\varrho)|e \in \mathcal C_{CP}(\varrho)\right)=E_j+F_j+G_j,$$
  where
  \begin{equation}
  \begin{split}
  E_j=\mathbb{P}\left(e_j\in \mathcal C_{CP}(\varrho)\Big{|}e \in \mathcal C_{CP}(\varrho),\, U_{e^+}<\frac{1}{1+\partial(e^+)\lambda_1}\right)\mathbb{P}\left(U_{e^+}<\frac{1}{1+\partial(e^+)\lambda_1}\right)
  \end{split}
  \end{equation}
  \begin{equation}
  \begin{split}
  F_j=\mathbb{P}\left(e_j\in \mathcal C_{CP}(\varrho)\Big{|}e \in \mathcal C_{CP}(\varrho),\, U_{e^+}\in \left[\frac{1+(i_{j}-1)\lambda_1}{1+\partial(e^+)\lambda_1}, \frac{1+i_{j}\lambda_1}{1+\partial(e^+)\lambda_1}\right]\right)\\
  \times \mathbb{P}\left( U_{e^+}\in \left[\frac{1+(i_{j}-1)\lambda_1}{1+\partial(e^+)\lambda_1}, \frac{1+i_{j}\lambda_1}{1+\partial(e^+)\lambda_1}\right]\right) 
  \end{split}
  \end{equation}
  \begin{equation}
  \begin{split}
 G_j=\mathbb{P}\left(e_j\in \mathcal C_{CP}(\varrho)\Big {|}e \in \mathcal C_{CP}(\varrho),\, U_{e^+}\in \bigcup_{i\in \{1,\cdots,\partial(e^+)\}\setminus\{i_{j}\}}\left[\frac{1+(i-1)\lambda_1}{1+\partial(e^+)\lambda_1}, \frac{1+i\lambda_1}{1+\partial(e^+)\lambda_1}\right]\right)\\
  \times \mathbb{P}\left( U_{e^+}\in \bigcup_{i\in\{1,\cdots,\partial(e^+)\} \setminus \{i_{j}\}}\left[\frac{1+(i-1)\lambda_1}{1+\partial(e^+)\lambda_1}, \frac{1+i\lambda_1}{1+\partial(e^+)\lambda_1}\right]\right).
  \end{split}
  \end{equation}
 
\begin{lemma}
\label{lem:quasiindependent}
There exists four constants $(\alpha_1, \alpha_2, \alpha_3, \alpha)$ depend on $\mathcal{T}$, $\lambda$ and $\lambda_1$ such that:
\begin{equation}
    \label{equa:quasiind1}
A\leq \alpha_1E_1E_2.
\end{equation}

\begin{equation}
    \label{equa:quasiind2}
B\leq \alpha_2F_1E_2.
\end{equation}

\begin{equation}
    \label{equa:quasiind3}
C\leq \alpha_3F_2E_1.
\end{equation}

\begin{equation}
    \label{equa:quasiind4}
D\leq \alpha_4G_1G_2.
\end{equation}
\end{lemma}

We deduce from Lemma~\ref{lem:quasiindependent} that 
$$A+B+C+D\leq \alpha(E_1+F_1+G_1)(E_2+F_2+G_2),$$
where $\alpha=\max_{i\in\{1,2,3,4\}}\alpha_i$. The latter inequality  concludes the proof of Proposition. 
\end{proof}

\begin{proof}[\textbf{Proof of Lemma~\ref{lem:quasiindependent}}]
Now, we will adapt the argument from the proof of Lemma 12 in \cite{collevecchio2017branching}. We prove that there exists $\alpha_1$ such that $A\leq \alpha_1E_1E_2$ and we use the same argument for the other inequalities.\\

First, by using condition \ref{condition1}, note that,
$$\mathbb{P}\left(U_{e^+}<\frac{1}{1+\partial(e^+)\lambda_1}\right)=\frac{1}{1+\partial(e^+)\lambda_1}\geq \frac{1}{1+M\lambda_1}, \text{ we then obtain: }$$

\begin{equation}
    \label{equa:c}
    \mathbb{P}\left(U_{e^+}<\frac{1}{1+\partial(e^+)\lambda_1}\right)\leq (1+M\lambda_1) \left[\mathbb{P}\left(U_{e^+}<\frac{1}{1+\partial(e^+)\lambda_1}\right)\right]^2.
\end{equation}
On the event $\left\{e \in \mathcal C_{CP}(\varrho),\, U_{e^+}<\frac{1}{1+\partial(e^+)\lambda_1}\right\}$ we have $X^{(e)}_{T^{(e)}(e^+)+1}=e^-$. We then define $\widetilde{T}^{(e)}(e^+):=\inf\left\{n\geq T^{(e)}(e^+)+1: X^{(e)}_n=e^+\right\}$.  We define the following quantities: 
 \begin{equation}
 \begin{split}
 N(e)&=\left|\left\{\widetilde{T}^{(e)}(e^+)\leq n \leq T^{(e)}(\varrho)\circ \theta_{\widetilde{T}^{(e)}(e^+)}: ({X}^{(e)}_n,{X}^{(e)}_{n+1})=(e^+,e^-)\right\} \right|,\\
 L(e)&=\sum_{j=0}^{N(e)-1} \frac{Y(e^+,e^-,j)}{r(e^+,e^-)},
 \end{split}
 \end{equation}
 where $|A|$ denotes the cardinality of a set $A$ and $\theta$ is the canonical shift on trajectories. Note that $L(e)$ is the time consumed by the clocks attached to the oriented edge $(e^+,e^{-})$ before ${\bf X}^{\ssup{e}}$, ${X}^{\ssup{e_1}}$ or ${X}^{\ssup{e_2}}$ goes back to $ {\r}$ once it has returned $e^+$ after the time $T^{(e)}(e^+)$. Recall that these three extensions are coupled and thus the time $L(e)$ is the same for the three of them.\\
 For $i\in \{1,2\}$, recall that $v_i$ is the vertex which is the offspring of $e^+$ lying the path from $\varrho$ to $e_i$. Note that $v_i$ could be equal to $e^+_i$. We define for $i\in \{1,2\}$: 

 \begin{equation}
 \begin{split}
 N^*(e_i)&=\left|\left\{\widetilde{T}^{(e)}(e^+)\leq n \leq T^{(e_i)}(e_i^+): ({X}^{[e^+, e_i^+]}_n,{X}^{[e^+,e_i^+]}_{n+1})=(e^+,v_i)\right\} \right|,\\
L^*(e_i)&=\sum_{j=0}^{N^*(e_i)-1}\frac{Y(e^+,e^-,j)}{r(e^+,e^-)}.
\end{split}
\end{equation}
Here, $L^*(e_i)$, $i\in\{1,2\}$, is the time consumed by the clocks attached to the oriented edge $(e^+,v_i)$ before ${\bf X}^{\ssup{e_i}}$, or ${\bf X}^{[e^+,e_i^+]}$,   hits $e_i^+$.\\
Notice  that the three quantities $L(e)$, $L^*(e_1)$ and $L^*(e_2)$ are independent, and we also have: 

\begin{equation}
\begin{split}
\mathbb{P}\left(e_1, e_2\in \mathcal C_{CP}(\varrho)\Big{|}e \in \mathcal C_{CP}(\varrho),\, U_{e^+}<\frac{1}{1+\partial(e^+)\lambda_1}\right)\\=\psi(e,\lambda)\mathbb{P}\left(L(e)>L^*(e_1)\vee L^*(e_2)\right).
\end{split}
\end{equation}

\begin{equation}
\begin{split}
\mathbb{P}\left(e_1\in \mathcal C_{CP}(\varrho)\Big{|}e \in \mathcal C_{CP}(\varrho),\, U_{e^+}<\frac{1}{1+\partial(e^+)\lambda_1}\right)=\psi(e,\lambda)\mathbb{P}\left(L(e)>L^*(e_1)\right).
\end{split}
\end{equation}

\begin{equation}
\begin{split}
\mathbb{P}\left(e_2\in \mathcal C_{CP}(\varrho)\Big{|}e \in \mathcal C_{CP}(\varrho),\, U_{e^+}<\frac{1}{1+\partial(e^+)\lambda_1}\right)=\psi(e,\lambda)\mathbb{P}\left(L(e)> L^*(e_2)\right).
\end{split}
\end{equation}

Now, the random variable $N(e)$ is simply a geometric random variable (counting the number of trials) with success probability $\lambda^{1-|e|}/\sum_{g\le e}\lambda^{1-|g|}$. 
The random variable $N(e)$ is independent of the family $Y(e^+,e^-,\cdot)$. As $Y(e^+,e^-,j)$ are independent exponential random variable for $j\ge0$, we then have that $L(e)$ is an exponential random variables with parameter 
 \begin{equation}\label{paramp}
 p:=\frac{\lambda^{1-|e|}}{\sum_{g\le e}\lambda^{1-|g|}}\times\lambda^{|e|-1}=\frac{1}{\sum_{g\le e}\lambda^{1-|g|}}.
 \end{equation}
 A priori, $L^{*}(e_1)$ and $L^{*}(e_2)$ are not exponential random variable, but they have a continuous distribution. Denote $f_1$ and $f_2$ respectively the densities of $L^{*}(e_1)$ and $L^{*}(e_2)$. Then, we have that
 \begin{equation}
 \begin{split}
\mathbb{P}\left(L(e)>L^*(e_1)\vee L^*(e_2)\right)&=\int_{0}^{+\infty}\int_{0}^{+\infty}\int_{x_1\vee x_2}^{+\infty}p\, e^{-pt}f_1(x_1)f_2(x_2)dtdx_1dx_2\\
&= \int_{0}^{+\infty}\int_{0}^{+\infty} e^{{-p}(x_1\vee x_2)}f_1(x_1)f_2(x_2)dx_1dx_2.\\
&\leq \int_{0}^{+\infty}\int_{0}^{+\infty} e^{\frac{-p}{2}(x_1+x_2)}f_1(x_1)f_2(x_2)dx_1dx_2.
 \end{split}
 \end{equation}
Thus, one can write
 \begin{equation}\label{qindep}
 \begin{split}
& \mathbb{P}\left(L(e)>L^*(e_1)\vee L^*(e_2)\right)\\
&\leq \left (\int_{0}^{+\infty} e^{-px_1/2}f_1(x_1)dx_1\right )\cdot\left (\int_{0}^{+\infty} e^{-px_2/2}f_2(x_2)dx_2\right ).
\end{split}
 \end{equation} 
 Note that: 
 \begin{equation}
 \int_{0}^{+\infty} e^{-px_1/2}f_1(x_1)dx_1=\mathbb{P}\left(\widetilde{L}(e)>L^*(e_1)\right),
 \end{equation}
 where $\widetilde{L}(e)$ is an exponential variable with parameter $p/2$. Note that, in view of \eqref{paramp}, $\widetilde{L}(e)$ has the same law as $L(e)$ when we replace the weight of an edge $g'$ by $\lambda^{-|g'|+1}/2$ for $g'\leq e$ only, and keep the other weights the same. \\
 For simplicity, for any $g\in E$, we set $w(g)=\lambda^{|g|-1}$. For $g\in E$ such that $e<g$, define the functions $\widetilde{\psi}$ and $\widetilde{\phi}$ in a similar way as $\psi$ and $\phi$, except that we replace the weight of an edge $g'$ by $\lambda^{-|g'|+1}/2$ for $g'\leq e$ only, and keep the other weights the same, that is, for $g\in E$, $e<g$,
 
\begin{align}
 \widetilde{\psi}(g,\lambda)=\frac{\sum_{g'<g}w(g')^{-1}+\sum_{g'\leq e}w(g')^{-1}}{\sum_{g'\leq g}w(g')^{-1}+\sum_{g'\leq e}w(g')^{-1}}.\\
  \widetilde{\phi}(g,\lambda_1,\lambda)= \frac{\lambda_1}{1+\partial(g^-)\lambda_1}+\frac{1}{1+\partial(g^-)\lambda_1}\widetilde{\psi}(g,\lambda)\widetilde{\psi}(g^{-1},\lambda)+\frac{(\partial(g^-)-1)\lambda_1}{1+\partial(g^-)\lambda_1}\widetilde{\psi}(g,\lambda).
\end{align}
We obtain: 
\begin{equation}
\begin{split}
&\mathbb{P}(\widetilde{L}(e)>L^*(e_1))=\prod_{e<g\leq e_1}\widetilde{\phi}(g, \lambda_1, \lambda)=\prod_{e<g\leq e_1}\phi(g, \lambda_1, \lambda)\prod_{e<g\leq e_1}\left(\frac{\widetilde{\phi}(g, \lambda_1, \lambda)}{\phi(g, \lambda_1, \lambda)}\right)\\
&=\mathbb{P}(L(e)>L^*(e_1))\times \prod_{e<g\leq e_1}\left(\frac{ \lambda_1+\widetilde{\psi}(g,\lambda)\widetilde{\psi}(g^{-1},\lambda)+(\partial(g^{-1})-1)\lambda_1\widetilde{\psi}(g,\lambda)}{\lambda_1+\psi(g,\lambda)\psi(g^{-1},\lambda)+(\partial(g^{-1})-1)\lambda_1\psi(g,\lambda)}\right)\\
&=\mathbb{P}(L(e)>L^*(e_1))\\
&\qquad\times \prod_{e<g\leq e_1}\left(1+\frac{ \widetilde{\psi}(g,\lambda)\widetilde{\psi}(g^{-1},\lambda)-\psi(g,\lambda)\psi(g^{-1},\lambda)+(\partial(g^{-1})-1)\lambda_1(\widetilde{\psi}(g,\lambda)-\psi(g,\lambda))}{\lambda_1+\psi(g,\lambda)\psi(g^{-1},\lambda)+(\partial(g^{-1})-1)\lambda_1\psi(g,\lambda)}\right).
\end{split}
\end{equation}
 
Now, we compute the product: 
$$ \prod_{e<g\leq e_1}\left(1+\frac{ \widetilde{\psi}(g,\lambda)\widetilde{\psi}(g^{-1},\lambda)-\psi(g,\lambda)\psi(g^{-1},\lambda)+(\partial(g^{-1})-1)\lambda_1(\widetilde{\psi}(g,\lambda)-\psi(g,\lambda))}{\lambda_1+\psi(g,\lambda)\psi(g^{-1},\lambda)+(\partial(g^{-1})-1)\lambda_1\psi(g,\lambda)}\right)$$

$$\leq  \prod_{e<g\leq e_1}\left(1+\frac{ \widetilde{\psi}(g,\lambda)\widetilde{\psi}(g^{-1},\lambda)-\psi(g,\lambda)\psi(g^{-1},\lambda)+(\partial(g^{-1})-1)\lambda_1(\widetilde{\psi}(g,\lambda)-\psi(g,\lambda))}{\lambda_1}\right).$$

$$\leq \exp{\left(\frac{1}{\lambda_1}\sum_{e<g\leq e_1}  \left(\widetilde{\psi}(g,\lambda)\widetilde{\psi}(g^{-1},\lambda)-\psi(g,\lambda)\psi(g^{-1},\lambda)+(\partial(g^{-1})-1)\lambda_1(\widetilde{\psi}(g,\lambda)-\psi(g,\lambda))\right)\right)} $$

\begin{lemma}\label{lem:abcd1}
There exists a constant $c=c(\lambda_1, \lambda)$ which do not depend on $e$, $e_1$ and $e_2$, such that: 
\begin{equation}
\sum_{e<g\leq e_1}\left(\widetilde{\psi}(g,\lambda)-\psi(g,\lambda)\right)\leq c. 
\end{equation}
\end{lemma}   
On the other hand, by using Lemma~\ref{lem:abcd1}, for any $e$ and $e_1$ we have that

\begin{equation}\label{equ:abcd2}
\sum_{e<g\leq e_1}\left(\widetilde{\psi}(g,\lambda)\widetilde{\psi}(g^{-1},\lambda)-\psi(g,\lambda)\psi(g^{-1},\lambda)\right)\leq 2c. 
\end{equation}

By using \ref{equ:abcd2}, Lemma~\ref{lem:abcd1} and condition \eqref{condition1}, we obtain:
\begin{equation}
\begin{split}
&\prod_{e<g\leq e_1}\left(1+\frac{ \widetilde{\psi}(g,\lambda)\widetilde{\psi}(g^{-1},\lambda)-\psi(g,\lambda)\psi(g^{-1},\lambda)+(\partial(g^{-1})-1)\lambda_1(\widetilde{\psi}(g,\lambda)-\psi(g,\lambda))}{\lambda_1+\psi(g,\lambda)\psi(g^{-1},\lambda)+(\partial(g^{-1})-1)\lambda_1\psi(g,\lambda)}\right)\\ 
&\leq \exp\left(Mc+\frac{2c}{\lambda_1}\right). 
\end{split}
\end{equation}

We have just proved that
\begin{equation}\label{qindep1}
\int_{0}^{+\infty} e^{-px_1/2}f_1(x_1)dx_1\leq \exp\left(Mc+\frac{2c}{\lambda_1}\right)\times \mathbb{P}(e_1\in \mathcal C_{CP}(\varrho)|e_1\wedge e_2 \in \mathcal C_{CP}(\varrho)).
\end{equation} 
By doing a very similar computation, one can prove that
\begin{equation}\label{qindep2}
\int_{0}^{+\infty} e^{-px_2/2}f_1(x_2)dx_2\leq \exp\left(Mc+\frac{2c}{\lambda_1}\right)\times \mathbb{P}(e_2\in \mathcal C_{CP}(\varrho)|e_1\wedge e_2 \in \mathcal C_{CP}(\varrho)).
\end{equation}

Moreover, we have
\begin{equation}\label{equ:abcdef1}
\psi(e,\lambda)\geq \frac{\lambda}{1+\lambda}.
\end{equation}

The conclusion \eqref{qindep0} follows by using \eqref{equa:c}, \eqref{qindep}, \eqref{equ:abcdef1},  together with \eqref{qindep1} and \eqref{qindep2}.
 \end{proof}

It remains to prove Lemma~\ref{lem:abcd1}.
\begin{proof}[\textbf{Proof of Lemma~\ref{lem:abcd1}}]
By a simple computation, for any $e<g\leq e_1$,
\begin{equation}\label{equ:aabbcc1}
\begin{split}
&\widetilde{\psi}(g,\lambda)-\psi(g,\lambda)=\frac{\left(\sum_{g'\leq e}w(g')^{-1}\right)w(g)^{-1}}{\left(\sum_{g'\leq g}w(g')^{-1}+\sum_{g'\leq e}w(g')^{-1}\right)\left(\sum_{g'\leq g}w(g')^{-1}\right)}.
\end{split}
\end{equation}
We will proceed by distinguishing three cases.\\
\noindent
{\bf Case I: $\lambda < 1$.}\\
By \eqref{equ:aabbcc1}, we have that 
\begin{equation}\label{equ:aabbcc2}
\begin{split}
&\widetilde{\psi}(g,\lambda)-\psi(g,\lambda)=\frac{\left(1-\frac{1}{\lambda^{|e|}}\right)\frac{1}{\lambda^{|g|-1}}}{\left(1-\frac{1}{\lambda^{|g|}}+1-\frac{1}{\lambda^{|e|}}\right)\left(1-\frac{1}{\lambda^{|g|}}\right)}\times \left(1-\frac{1}{\lambda}\right). 
\end{split}
\end{equation}
Hence, there exists a constants $c_1$ such that 
\begin{align}
0\leq \widetilde{\psi}(g,\lambda)-\psi(g,\lambda)\leq c_1 \lambda^{|g|-|e|}.
\end{align}

Therefore we obtain
\begin{equation}
\sum_{e<g\leq e_1}\left(\widetilde{\psi}(g,\lambda)-\psi(g,\lambda)\right)\leq c_1 \sum_{e<g\leq e_1}\lambda^{|g|-|e|}\leq c_1\sum_{i\geq 0}\lambda^i<\infty.
\end{equation}

\noindent

{\bf Case II: $\lambda=1$.}\\
By \eqref{equ:aabbcc1}, we have that 

\begin{equation}
\widetilde{\psi}(g,\lambda)-\psi(g,\lambda)=\frac{|e|}{|g|(|g|+|e|)}.
\end{equation}
Therefore we obtain
\begin{equation}
\label{equ:aabbccdd1}
\begin{split}
&\sum_{e<g\leq e_1}\left(\widetilde{\psi}(g,\lambda)-\psi(g,\lambda)\right)\leq \sum_{n\geq |e|}\left(\frac{|e|}{n(n+|e|)}\right)\leq \sum_{n\geq |e|}\left(\frac{1}{n}-\frac{1}{n+|e|}\right)\\
&\leq \sum_{n=|e|}^{2|e|-1}\frac{1}{n}. 
\end{split}
\end{equation}

On the other hand, we have:

\begin{equation}
\label{equ:aabbccdd2}
\lim_{n\rightarrow \infty}\left(\sum_{k=n}^{2n-1}\frac{1}{k}\right)=\lim_{k\rightarrow \infty}\left(\sum_{k=0}^{n-1}\frac{1}{n+k}\right)=\lim_{k\rightarrow \infty}\left(\frac{1}{n}\sum_{k=0}^{n-1}\frac{1}{1+k/n}\right)=\int_{0}^1\frac{dx}{1+x}.
\end{equation}
We use \eqref{equ:aabbccdd1} and \eqref{equ:aabbccdd2} to obtain the result. 

\noindent

{\bf Case III: $\lambda>1$.}\\
By \eqref{equ:aabbcc1}, we have that 
\begin{equation}\label{equ:aabbcc2}
\begin{split}
&\widetilde{\psi}(g,\lambda)-\psi(g,\lambda)=\frac{\left(1-\frac{1}{\lambda^{|e|}}\right)\frac{1}{\lambda^{|g|-1}}}{\left(1-\frac{1}{\lambda^{|g|}}+1-\frac{1}{\lambda^{|e|}}\right)\left(1-\frac{1}{\lambda^{|g|}}\right)}\times \left(1-\frac{1}{\lambda}\right). 
\end{split}
\end{equation}
Hence, there exists a constants $c_2$ such that 
\begin{align}
0\leq \widetilde{\psi}(g,\lambda)-\psi(g,\lambda)\leq \frac{c_2}{\lambda^{|g|}}. 
\end{align}

Therefore we obtain
\begin{equation}
\sum_{e<g\leq e_1}\left(\widetilde{\psi}(g,\lambda)-\psi(g,\lambda)\right)\leq c_2 \sum_{e<g\leq e_1}\frac{1}{\lambda^{|g|}}\leq c_2\sum_{i\geq 0}\left(\frac{1}{\lambda}\right)^i<\infty.
\end{equation}

\end{proof}

\subsection{Transience in Theorem \ref{maintheorem}: The case $RT(\mathcal{T},{\bf X})>1$}

\begin{proposition}\label{proptrans}
If $RT(\mathcal{T},\X)>1$ and if \eqref{condition1} is satisfied  then  ${\bf X}$ is transient.
\end{proposition}

\begin{proof}
The proof is now easy, we can follow line by line the Appendix A.2 of  \cite{collevecchio2018branching}.
\end{proof}

\section{Proof of Theorem~\ref{thm:criticaldigging}}
This section is independent with the previous sections. In this section, we prove a criterion which can apply to the critical $M$-digging random walk on superperiodic trees. We will use the Rubin's construction (resp.\ the definition of $\mathcal C(\varrho)$, $\mathcal C_{CP}(\varrho)$) from section 7 (resp.\ section 8.1) of \cite{collevecchio2018branching}. We will allow ourselves to omit these definitions and refer the readers to \cite{collevecchio2018branching} for more details.\\

The main idea for the proof of Theorem~\ref{thm:criticaldigging} is that the number of surviving rays of the percolation $\mathcal C_{CP}(\varrho)$ almost surely is either zero or infinite. This property was proved in the case of \emph{Bernoulli percolation} (see \cite{LP:book} proposition 5.27) or \emph{target percolation} (see \cite{pemantle1995critical}, lemma 4.2). The main difficulty that we have to face is that the FKG inequality is not true for our percolation. \\

\subsection{Some definitions}
Let $\lambda>0$, $M\in \mathbb N$ and $\mathcal{T}$ be an infinite, locally finite and rooted tree. For each $v \in V(\mathcal{T})$,  recall the definition of subtree $\mathcal{T}^v$ of $\mathcal T$ from Section~\ref{sub:notation}. Let ${\bf X}^{v, \lambda}$ be the
  $M$-digging random walk  on $\mathcal{T}^v$. We say that $\mathcal{T}$ is \emph{uniformly transient} if for any $\lambda$ such that the $M$-digging random walk on $\mathcal{T}$ with parameter $\lambda$ is transient (i.e.\ ${\bf X}^{\varrho, \lambda}$ is transient),
  \begin{equation}
  \exists \alpha_{\lambda}>0,\forall  v
  \in    V(\mathcal{T}), \mathbb{P}(\forall     n>0,    X^{v, \lambda}_{n}     \neq    v)\geq
  \alpha_{\lambda}.
  \end{equation}

It is called \emph{weakly uniformly transient} if there exists a sequence of finite pairwise disjoint $\pi_n$ such that
\begin{equation}
\exists \alpha_{\lambda}>0,\forall  v \in\underset{n}{\bigcup}V(\pi_n), \mathbb{P}(\forall     n>0,    X^{v, \lambda}_{n}     \neq    v)\geq
  \alpha_{\lambda}
\end{equation}
where $V(\pi_n)=\{e^-: e\in \pi_n\}$. 

\begin{remark}
\begin{itemize}
\item If $\mathcal{T}$ is uniformly transient, then $\mathcal{T}$ is also weakly uniformly transient, but the reverse is not always true.

\item The superperiodic trees are uniformly transient. 
\end{itemize}

\end{remark}

\bigskip

An infinite  self-avoiding path starting at  $\varrho$ is called a  \emph{ray}. The set of  all rays,  denoted   by  $\partial   \mathcal{T}$,  is   called  the \emph{boundary} of  $\mathcal T$. Let $\phi: \mathbb Z^{+}\rightarrow \mathbb{R}$ be a decreasing positive function with $\phi(n) \rightarrow 0$ as $n \rightarrow \infty$. The \emph{Hausdorff mearsure} of $\mathcal{T}$ in gauge $\phi$ is

 $$\liminf_{\Pi}\sum_{v\in \Pi}\phi(|v|),$$
where the $\liminf$ is taken over $\Pi$ such that the distance from $\varrho$ to the nearest vertex  in  $\Pi$  goes  to  infinity. We say that $\mathcal{T}$ has $\sigma$-finite Hausdorff measure in gauge $\phi$ if $\partial \mathcal{T}$ is the union of countably many subsets with finite Hausdorff measure in gauge $\phi$. 

Finally, If $\lambda$ is such that the $M$-digging random walk $X$ with parameter $\lambda$ on $\mathcal{T}$ is transient, on the event $\{T(\varrho)=\infty\}$, its path determines an infinite branch in $\mathcal{T}$, which can be seen as a random ray $\omega^{\infty}$, and call it the \emph{limit walk} of $X$. Equivalently, on the event $\{T(\varrho)=\infty\}$, we define the limit walk as follows: For any $k\geq 1$,
\begin{equation}
  \omega^{\infty}(k)=v \iff  v \in \mathcal{T}_k  \text{~and~} \exists
  n_0, \forall n>n_0: X_n \in \mathcal{T}^{v}.
\end{equation}

Note that $\mathbb{P}\left(\omega^{\infty}(0)=\varrho\right)=1$. For any $k\geq 1$, we call the $k$-first steps of $\omega^{\infty}$ is $(\omega^{\infty}(0), \cdots, \omega^{\infty}(k))$, denoted by $\omega^{\infty}_{|{[0,n]}}$.

\subsection{Proof of Theorem~\ref{thm:criticaldigging}} $ $

\begin{figure}[h!]
    \centering
    \includegraphics[scale=0.5]{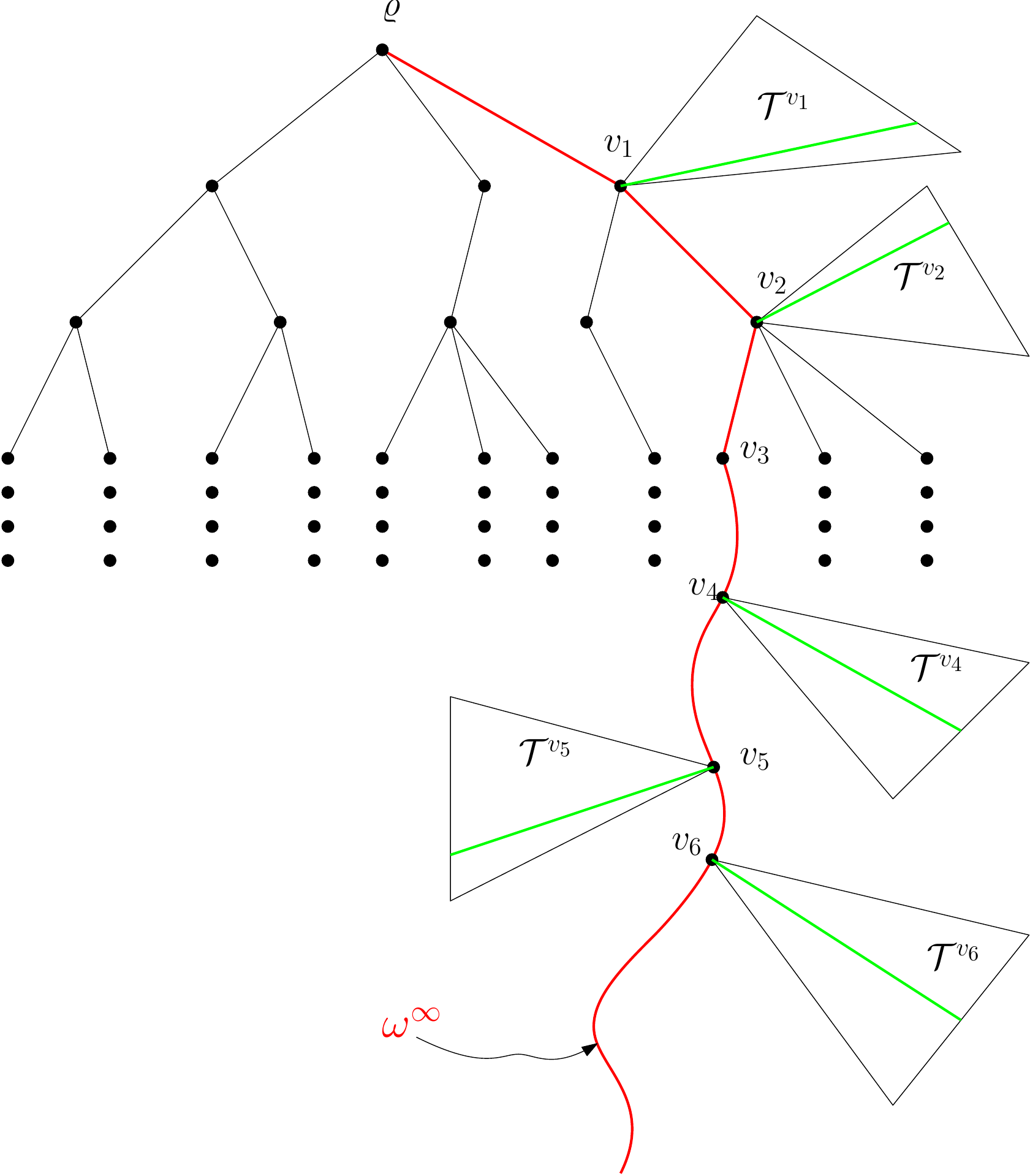} 
    \caption{The proof's idea of Proposition~\ref{pro:criticalcase2}. The limit walk $\omega^{\infty}$ is in red. Conditioning on the event $\{\omega^\infty(0)=\varrho, \omega^\infty(1)=v_1,...,\omega^\infty(6)=v_6\}$ and denote by $\ell$ the last time the critical $M$-digging random walk ${\bf X}$ on $\mathcal T$ visits $v_6$. For each $1\leq i\leq 6$, running the walk ${\bf X}^{v_i, \lambda_c}$ on $\mathcal{T}^{v_i}$. The property of uniformly transient implies that there exists a surviving ray (in green) in $\mathcal{T}^{v_i}$ with probability is larger than a constant which do not depend on $i$.}
    \label{fig:abcdefgh11}
  \end{figure}
We begin with the following proposition:

\begin{proposition}
\label{pro:criticalcase2}
Let $\bf X$ be a $M$-digging random walk with parameter $\lambda_c$ on an uniformly transient tree $\mathcal T$ and recall the definition of $\mathcal{C}_{CP}$ from $\bf X$ as in (\cite{collevecchio2018branching}, Section 7). Consider the percolation induced by $\mathcal{C}_{CP}$ and let $\phi(n)=\mathbb P(\varrho \leftrightarrow v)$ for $v\in \mathcal{T}_n$. 

\begin{enumerate}
    \item Almost surely, the number of surviving rays is either zero or infinite.
    
    \item If $\partial \mathcal{T}$ has $\sigma$-finite Hausdorff measure in the gauge $\{\phi(n)\}$, then $\mathbb P(\varrho \leftrightarrow \infty)=0$. In particular, $\bf X$ is recurrent. 
\end{enumerate}

\end{proposition}

The overall strategy for the proof of Proposition~\ref{pro:criticalcase2} is as follows. First, if ${\bf X}$ is recurrent, then the percolation induced by $\mathcal{C}_{CP}$ almost surely have no surviving ray. Next, assume that ${\bf X}$ is transient. On the event $\left\{T(\varrho)=\infty\right\}$, the limit walk $\omega^{\infty}$ is a surviving ray of $\mathcal{C}_{CP}(\varrho)$. Given $n \in \mathbb{N}$ and conditioning on $\omega^{\infty}_{|{[0,n]}}$, by using the Rubin's construction and the definition of uniformly transient, we prove that there exists a surviving ray in $\mathcal{T}^{\omega^{\infty}(i)}$ with probability larger than a constant $c$ which do not depend on $i$ and $\omega ^{\infty}$ (see Figure~\ref{fig:abcdefgh11}). The following basic lemma is necessary:

\begin{lemma}
\label{lem:comparecookies}
Let $\lambda>0$ and $\mathcal{T}$ be an infinite, locally finite and rooted tree. Let $\overline{M}:=(m_v, v\in V(\mathcal T))$ be a family of non-negative integers. Denote by $\bf X$ the $M$-digging random walk with parameter $\lambda$ and $\bf Y$
the $\overline{M}$-digging random walk associated with the inhomogeneous initial number of cookies $\overline{M}$ with parameter $\lambda$ (see \cite{collevecchio2018branching}, section 2.3.2 for more details on the definition of $\overline{M}$-digging random walk). Denote by $T^{\bf X}(\varrho)$ (resp. $T^{\bf Y}(\varrho)$) the return time of $\bf X$ (resp. $\bf Y$) to $\varrho$. Assume that $m(v)\leq M$ for all $v\in V(\mathcal{T})$, we then have
\begin{equation}
\mathbb{P}\left(T^{\bf X}(\varrho)<\infty\right)\leq \mathbb{P}\left(T^{\bf Y}(\varrho)<\infty\right).
\end{equation}
\end{lemma}
\begin{proof}
The proof is simple, therefore it is omitted. 
\end{proof}

\begin{proof}[Proof of Proposition~\ref{pro:criticalcase2}]
Let $\mathcal A_k$ denote  the  event that exactly $k$ rays survive and assume that 
\begin{equation}
\mathbb P(\mathcal A_k)>0,
\end{equation}
Hence, 

\begin{equation}
\label{equ:prove1}
P(|\mathcal{C}_{CP}(\varrho)|=\infty)>0.
\end{equation}

By \eqref{equ:prove1} and Lemma 22 in \cite{collevecchio2018branching}, we have that:
\begin{equation}
\mathbb P(T(\varrho)=\infty)>0,
\end{equation}
and therefore $\bf X$ is transient. 

On the event $\{T(\varrho)=\infty\}$, the limit walk $\omega^{\infty}$ of $\bf X$ is well defined and it is a surviving ray. Let $n$ be a positive integer and $\gamma:=(\gamma_0=\varrho, \gamma_1=v_1, \cdots, \gamma_n=v_n)$ be a path of length $n$ of $\mathcal{T}$. Denote by $B_{n, \gamma}$ the following event: 
\begin{equation}
\mathcal B_{n, \gamma}:=\{\omega^{\infty}_{|{[0,n]}}=\gamma\}.
\end{equation}

For any $1\leq k \leq n$, define a sub-tree $\mathcal{T}^{v_i}$ of $\mathcal{T}$ in the following way (see Figure~\ref{fig:abcdefgh11}). 

\begin{itemize}
\item The root of $\mathcal{T}^{v_i}$ is the vertex $v_i$. 

\item If $\partial (v_i)<2$ then $\mathcal{T}^{v_i}$ is a tree with a single vertex $v_i$: for example, $\mathcal{T}^{v_3}$ in Figure~\ref{fig:abcdefgh11}. 

\item  If $\partial (\gamma_i)\geq 2$, choose one of its children which is different to $v_{i+1}$, denoted by $v$. We then set:
$$\left\{\begin{matrix}
(\mathcal T^{v_i})_1=\{v\}\\ 
(\mathcal T^{v_i})^{v}=\mathcal T^v
\end{matrix}\right.$$

\end{itemize}

Note that for every pair $(i,j)\in [1,n]^2$, we have $V(\mathcal{T}^{v_i}) \cap V(\mathcal{T}^{v_j})=\varnothing$. 
\bigskip

Now, conditioning on the event $\mathcal B_{n, \gamma}$. Let $\ell$ be the last time ${\bf X}$ visits $v_n$, i.e.\ 

\begin{equation}
\ell:= \sup\{k>0: X_k=v_n\}.
\end{equation}
By the definition of limit walk, $\ell$ is finite on the event $\mathcal B_{n, \gamma}$. For each $i\in [1,n]$ and for all $v\in V(\mathcal{T}^{v_i})$, denote by $m^i(v)$ the remaining number of cookies at $v$ after time $\ell$, i.e.
\begin{equation}
m^i(v):=M- \#\{k\leq \ell: X_k=v\}. 
\end{equation}

By using the extensions introduced in (\cite{collevecchio2018branching}, Section 7), the next steps on the tree $\mathcal{T}^{v_i}$ are given by the digging random walk associated with the inhomogeneous initial number of cookies $(m^i(v), v\in V(\mathcal{T}^{v_i}))$ and the same parameter $\lambda_c$ as $\bf X$, denoted by ${\bf{X}}^{v_i, m^i, \lambda_c}$ (see \cite{collevecchio2018branching}, section 2.3.2 for more details on the definition of ${\bf{X}}^{v_i, m^i, \lambda_c}$).  Denote by $T^{v_i, m^i, \lambda_c}$ the return time of ${\bf{X}}^{v_i, m^i, \lambda_c}$ to the root $v_i$ of  $\mathcal{T}^{v_i}$. By the definition of uniformly transient and Lemma~\ref{lem:comparecookies}, there exists a constant $c>0$ which do not depend on $n$ and $\gamma$ such that for any $i$, 
\begin{equation}
\label{equ:prove2}
\mathbb{P}\left(T^{v_i, m^i, \lambda_c}<\infty\right)>c. 
\end{equation}

On the event $\{T^{v_i, m^i, \lambda_c}<\infty\}$, note that $\mathcal{C}_{CP}$ contains a surviving ray in $\mathcal{T}^{v_i}$. By \eqref{equ:prove2}, we have
\begin{equation}
\label{equ:prove3}
\mathbb{P}(\mathcal A_k| \mathcal B_{n, \gamma})\leq \binom{n}{k}(1-c)^{n-k}
\end{equation}

On the other hand,  we have $\mathcal A_k\subset \bigcup_{\gamma: |\gamma|=n}\mathcal B_{n, \gamma}$, therefore by \eqref{equ:prove3} we obtain:

\begin{equation}
\label{equa:A_1}
\mathbb{P}(\mathcal A_k)=\sum_{\gamma: |\gamma|=n}\mathbb{P}(\mathcal A_k| \mathcal B_{n, \gamma})\times \mathbb{P}(\mathcal B_{n, \gamma})\leq \left(\sum_{i=1}^{k}\binom{n}{i}\right)(1-c)^n \underbrace{\sum_{\gamma: |\gamma|=n}\mathbb{P}(\mathcal B_{n, \gamma})}_{\leq 1}\leq \left(\sum_{i=1}^{k}\binom{n}{i}\right)(1-c)^n.
\end{equation}

Since \ref{equa:A_1} holds for any $n$ then we obtain the following contradiction
\begin{equation}
\mathbb{P}(\mathcal A_k)=0.
\end{equation}

For part (2), the proof is similar to part (ii), Lemma 4.2 in \cite{pemantle1995critical}. 
\end{proof}

In the same method as in the proof of Proposition~\ref{pro:criticalcase2}, we can prove the slightly stronger result (the proof of which we omit):

\begin{proposition}
\label{pro:criticaldiggingcase3}
Let $\bf X$ be a $M$-digging random walk with parameter $\lambda_c$ on a weakly uniformly transient tree $\mathcal T$ and recall the definition of $\mathcal{C}_{CP}$ from $\bf X$ as in (\cite{collevecchio2018branching}, Section 7). Consider the percolation induced by $\mathcal{C}_{CP}$ and let $\phi(n)=\mathbb P(\varrho \leftrightarrow v)$ for $v\in \mathcal{T}_n$. 

\begin{enumerate}
    \item With probability one, the number of surviving rays is either zero or infinite.
    
    \item If $\partial \mathcal{T}$ has $\sigma$-finite Hausdorff measure in the gauge $\{\phi(n)\}$, then $\mathbb P(\varrho \leftrightarrow \infty)=0$. In particular, $\bf X$ is recurrent. 
\end{enumerate}

\end{proposition}

The following corollary is an immediate consequence of Proposition~\ref{pro:criticaldiggingcase3}. 

\begin{corollary}\label{cor:reccrit}
Let $M\in \mathbb{N}$ and $\mathcal{T}$ be a weakly uniformly transient tree such that $\partial \mathcal T$ has $\sigma$-finite Hausdorff measure in the gauge $\{\phi(n)\}=\left(\frac{1}{br(\mathcal{T})}\right)^n$ if $br(\mathcal{T})>1$ and $\{\phi(n)\}=\frac{1}{n^{M+1}}$ if $br(\mathcal{T})=1$. Then the critical $M$-digging random walk on $\mathcal{T}$ is recurrent.
\end{corollary}

\begin{proposition}
Let $M\in \mathbb{N}^*$ and $\mathcal{T}$ be a superperiodic tree whose upper-growth rate is finite. The critical $M$-digging random walk on $\mathcal{T}$ is recurrent. 
\end{proposition}
\begin{proof}
This is a consequence of Corollary~\ref{cor:reccrit} and Theorem~\ref{sousperiodic}.
\end{proof}

\begin{remark}
If $M=0$, then $M$-DRW$_{\lambda}$ is the biased random walk with parameter $\lambda$. The recurrence of critical biased random walk on $\mathcal{T}$ is a consequence of Theorem~\ref{sousperiodic} and Nash-Williams criterion (see  \cite{LP:book} or \cite{Nash1959}). 
\end{remark}

\begin{ack}
I am grateful to Daniel Kious and Andrea Collevecchio for precious discussions. I am thankful to Vincent Beffara for his support. 
\end{ack}

\end{document}